\documentclass[11pt]{amsart}
\usepackage{amssymb}
\usepackage{dsfont}
\usepackage{color}

\usepackage[utf8]{inputenc}
\usepackage[T1]{fontenc}
\renewcommand\labelenumi{(\roman{enumi})}
\renewcommand\theenumi\labelenumi

\newtheorem{theorem}{Theorem}[section]

\newtheorem{lemma}[theorem]{Lemma}
\newtheorem{proposition}[theorem]{Proposition}

\newtheorem{corollary}[theorem]{Corollary}

\newtheorem{definition}[theorem]{Definition}
\newtheorem{remark}[theorem]{Remark}
\newcommand{\B}{\mathcal B}
\newcommand{\mC}{\mathcal C}
\newcommand{\D}{\mathcal D}

\newcommand{\mX}{\mathcal X}
\newcommand{\A}{\mathcal A}
\newcommand{\N}{\mathbb N}

\newcommand{\Q}{\mathbb Q}

\newcommand{\R}{\mathbb R}

\newcommand{\Z}{\mathbb Z}

\newcommand{\on}{\operatorname}

\newcommand{\Aut}{\operatorname{Aut}}
\newcommand{\Ism}{\operatorname{Ism}}

\newcommand{\dom}{\operatorname{dom}}
\newcommand{\img}{\operatorname{rng}}
\newcommand{\rng}{\operatorname{rng}}
\author{Szymon G\l \c ab}

\address{Institute of Mathematics, \L\'od\'z University of Technology, W\'olcza\'nska 215, 93-005
\L\'od\'z, Poland}
\email {szymon.glab@p.lodz.pl}

\author{Przemys\l aw Gordinowicz}
\address{Institute of Mathematics, \L\'od\'z University of Technology, W\'olcza\'nska 215, 93-005
\L\'od\'z, Poland}
\email {pgordin@p.lodz.pl}

\author{Filip Strobin}

\address{Institute of Mathematics, \L\'od\'z University of Technology, W\'olcza\'nska 215, 93-005
\L\'od\'z, Poland}
\email {filip.strobin@p.lodz.pl}

\title[Dense free subgroups of automorphism groups]{Dense free subgroups of automorphism groups of homogeneous partially ordered sets}
\date{}

\begin{document}

\begin{abstract}
A countable poset is \emph{ultrahomogeneous} if every isomorphism between its finite subposets can be extended to an {\color{black}automorphism. The} groups $\Aut(A)$ of such posets $A$ have a natural topology in which $\Aut(A)$ are Polish topological groups. We consider the problem whether $\Aut(A)$ contains a dense free subgroup of two generators. We show that if $A$ is ultrahomogeneous, then $\Aut(A)$ contains such subgroup. Moreover we characterize whose countable ultrahomogeneous posets $A$ such that for each natural $m$, the set of all cyclically dense elements $\overline{g}\in\Aut(A)^m$ for the diagonal action is comeager in $\Aut(A)^m$. {\color{black}In our considerations we strongly use the result of  Schmerl which says that there are essentially four types of countably infinite ultrahomogeneous posets.} 
\end{abstract}

\maketitle

\section{Introduction}

In this paper we consider dense free subgroups of two generators of a Polish group of the order--preserving automoprhisms of a countable ultrahomogeneous partially ordered sets. We say that a Polish group with a dense (free) subgroup of two generators is \textit{topologically (freely) 2-generated}. In 1977 McDonough \cite{McD} proved that the group $S_\infty$ of all permutations on the natural numbers is topologically freely 2-generated; for more references and further improvements of this result see Darji and Mitchell \cite{DM1}. Among Polish groups that are topologically 2-generated there are the automorphism group $\Aut(\Q,\leq)$ of the rationals and the automorphism group $\Aut(\mathcal{R})$ of the random graph \cite{DM2}.   

An element $g$ in a Polish group $G$ is called {\color{black}\emph{ cyclically dense}} if $\{f^kgf^{-k}:k\in\Z\}$ is dense in $G$. Note that the existence of a cyclically dense element implies that $G$ is topologically 2-generated (an example of a group which is topologically 2-generated but does not have a cyclically dense element is given in Remark \ref{BnHasNoCyclicallyDenseElements}). The following generalization of a cyclically dense element was introduced by Solecki in \cite{Solecki}. The action
\[
G\times G^m\ni(f,\bar{g})\mapsto(fg_1f^{-1},\dots, fg_mf^{-1})\in G^m
\]
is called a \emph{diagonal action} of $G$ on $G^m$. We say that $\bar{g}\in G^m$ is \emph{cyclically dense} for the diagonal action of $G$ on $G^m$ if for some $f\in G$, $\{(f^kg_1f^{-k},\dots, f^kg_mf^{-k}):k\in\Z\}$ is dense in $G^m$. Solecki proved  in \cite{Solecki} that the isometry group $\on{Iso}(\mathbb{U}_0)$ of the rational Urysohn metric space $\mathbb{U}_0$ has  cyclically dense elements for each diagonal action. Consequently $\on{Iso}(\mathbb{U}_0)$ is topologically 2-generated.

A Polish group $G$ has {\emph{the Rokhlin property}} if if it has a dense conjugacy class. Clearly, the existence of a cyclically dense element in $G$ implies that $G$ has the Rokhlin property. A Polish group $G$ has {\emph{the strong Rokhlin property}} if it has a comeager conjugacy class. Clearly, the strong Rokhlin property implies Rokhlin property. The reverse implication does not hold. Solecki proved that the isometry group $\on{Iso}(\mathbb{U})$ of the Urysohn metric space $\mathbb{U}$ has  cyclically dense elements for all diagonal actions, which implies that $\on{Iso}(\mathbb{U})$ has the Rokhlin property. On the other hand Kechris proved that each conjugacy class of $\on{Iso}(\mathbb{U})$ is meager and consequently it does not have strong Rokhlin property; for the proof see \cite{GW}.

In their paper \cite{KR} Kechris and Rosendal {\color{black}defined} even stronger property than the strong Rokhlin property. A Polish group $G$ has {\color{black}\emph{ample generics}} if for each finite $n$ there is a comeager orbit for the diagonal action of $G$ on $G^n$. Solecki proved in \cite{Solecki} that $\on{Iso}(\mathbb{U}_0)$ has ample generics. The strong Rokhlin property does not imply that a given group has ample generics -- the automorphism group $\on{Aut}(\Q,\leq)$ of rationals has the strong Rokhlin property \cite{T1992} but does not have ample generics \cite{T2007}.

Gartside and Knight in \cite{GK} presented several consequences of the fact that a Polish non-Abelian and non-discrete group contains dense free subgroup. In such case almost all finitely generated groups are free, almost all countably generated groups are free and almost all compactly generated groups are free -- see \cite{GW} for precise definitions. In particular, it implies that the set {\color{black}$$\{(f,g)\in G^2:\{f,g\}\mbox{ freely generates a free subgroup of }G\}$$}is co-meager in $G^2$. We prove in Proposition \ref{PropNowhereDense} that in cases we are particularly interested on the set $\{(f,g)\in G^2:\{f,g\}$ freely generates a \textit{dense} free subgroup of $G\}$ nowhere dense in $G^2$.

Recently Jonu\v{s}as and Mitchell in \cite{JM} showed that the automorphism groups of countable ultrahomogeneous graphs are topologically 2-generated. By the characterization of countable ultrahomogeneous graphs given by Woodrow and Lachlan \cite{LW}, they need only to consider four types {\color{black}of} automorphism groups. There is a similar characterization of countable ultrahomogeneous posets given by Schmerl \cite{Sch}, see Theorem \ref{SchmerlTheorem} below. This has inspired us to study the similar problem to that of Jonu\v{s}as and Mitchell for countable ultrahomogeneous posets.\\
{\color{black} For $n\in\N$, let $A_n:=\{1,...,n\}$ and additionally put $A_\omega:=\{1,2,3,...\}$.  Choose $1\leq n\leq\omega$.\\
By $\mathcal{A}_n$ we denote the partially order set $(A_n,\leq)$, where "$\leq$" is the trivial partial order, i.e., $x\leq y$ iff $x=y$.\\
By $\B_n$ we denote the partially ordered set $(A_n\times\mathbb{Q},\leq)$, where $\mathbb{Q}$ is the set of rational numbers and the partial order $\leq$ is defined by $(k,p)\leq (m,q)$ iff $k=m$ and $p\leq q$. We refer to $\B_n$ as \emph{an antichain of chains}.\\
By $\mC_n$ we denote partially ordered set $(A_n\times\mathbb{Q},\leq)$, where the partial order $\leq$ is defined $(k,p)\leq (m,q)$ iff $p\leq q$. We refer to $\mC_n$ as \emph{a chain of antichains}.\\
Finally, let $\D$} be the generic (universal countable homogeneous) partially ordered set, that is a Fra\"iss\'e limit of all finite partial orders (in Section~\ref{SectionD} we explain the definition of $\D$).

Schmerl~\cite{Sch} showed that there are only countably many, up to isomorphism, ultrahomogeneous countable partially ordered sets. More precisely he proved the following characterisation.
\begin{theorem}\label{SchmerlTheorem}
Let $(H,<)$ be a countable partially ordered set. Then $(H,<)$ is ultrahomogeneous iff it is isomorphic to one of the following:
\begin{itemize}{\color{black}
\item[(a)] $\A_n$ for $1\leq n\leq \omega$;
\item[(b)] $\B_n$ for $1\leq n\leq \omega$;
\item[(c)] $\mC_n$ for $2\leq n\leq \omega$;
\item[(d)] $\D$.}
\end{itemize}
Moreover, no two of the partially ordered sets listed above are isomorphic.
\end{theorem}

Consider automorphisms groups $\operatorname{Aut}(\A_\omega)=S_\infty$, $\operatorname{Aut}(\B_n)$, $\operatorname{Aut}(\mC_n)$ and $\operatorname{Aut}(\D)$. We proved that each of these groups contains two elements $f,g$ such that the subgroup generated by $f$ and $g$ is free and dense. By Schmerl's Theorem, this proves our main result:
\begin{theorem}\label{main}
Let $(H,<)$ an ultrahomogeneous countable partially ordered set. Then $\Aut(A)$ is freely topologically 2-generated.
\end{theorem}
{\color{black}In fact, we prove} something more: at each case $A=\A_n,\B_n,\mC_n,\D$, one can find a certain Polish subset $\mathcal{X}\subseteq \Aut(A)\times\Aut(A)$ such that the set of all pairs 
$$\{(f,g)\in\mathcal{X}: \mbox{$f$ and $g$ freely generate a dense subgroup of }\Aut(A)\}$$
is comeager in $\mathcal{X}$.
Moreover, in the case of $A=\A_n,\B_\omega,\mC_n,\D$, we show that for each $m$ the set of all cyclically dense elements $\overline{g}\in\Aut(A)^m$ for the diagonal action is comeager in $\Aut(A)^m$.

\section{Preliminaries}

We use standard set-theoretic notation, see \cite{Gao,kechris}. By $\omega$ we denote the first infinite ordinal number, which we identify with the set of natural numbers $\omega=\{0,1,2,\dots\}$. 

Let $A$ be a countable structure such that its finitely generated substructures are finite (this is true for relational structures which are of our interest here). By $\Aut(A)$ we denote the automorphism group of $A$. A function $f:X\to Y$ which is an isomorphism between two finite substructures $X$ and $Y$ of $A$ is called \emph{partial isomorphism of $A$}. The set of all partial isomorphisms is denoted by $\Ism(A)$. A countable structure $A$ is called \emph{ultrahomogeneous}, if every partial isomorphism $f\in\Ism(A)$ can be extended to an automorphism $\bar{f}\in\Aut(A)$.

Since $A$ is countable, the automorphism group $\Aut(A)$ can be viewed as a subgroup of the symmetric group $S_\infty$ of natural numbers. We consider the usual topology on $\Aut(A)$, inherited from $S_\infty$, generated by the basic sets of the form $\{\bar{f}\in\Aut(A): f\subset\bar{f}\}$ where $f\in\Ism(A)$. It turns out that this is a Polish topology, i.e. completely meatrizable and separable, as $\Aut(A)$ is a closed subgroup of $S_\infty$. For details see for example \cite{Gao}.  

By a \emph{word of letters $a,b$}, we mean each sequence of the form
\begin{equation}\label{words}
w(a,b)=c_1^{n_1}c_2^{n_2}c_3^{n_3}...c_k^{n_k}
\end{equation}
where $n_1,...,n_{k}\in\Z$ and $c_1,...,c_k\in\{a,b\}$. A word $w(a,b)$ of the form \eqref{words} is called \emph{irreducible}, if $n_1,...,n_k\neq 0$ and $c_i\neq c_{i+1}$ for $i=1,...,k-1$. By the \emph{length} $\vert w\vert$ of a word $w$ we mean the value $n_1+...+n_k$.

An automorphism group $\Aut(A)$ is \emph{freely topologically 2-generated} if it contains two elements $f$ and $g$ which freely generate a dense subgroup of $\Aut(A)$, i.e. two functions ${f},{g}$ such that:
\begin{itemize}
\item  for every irreducible word $w(a,b)$, the natural automorphism $w(f,g)$ is not the identity function $\on{id}$;
\item the subgroup $\langle f,g\rangle=\{w(f,g):w\mbox{ is a word}\}$ is dense in $\Aut(A)$.
\end{itemize} 

A set $M$ in a Polish space $X$ is called { \emph{meager}} if $M$ { is} a countable union of nowhere dense subsets of $X$. A set $C$ in $X$ is called {\emph{comeager}} if it is a complement of a meager set. By { the} Baire category theorem comeager sets are non-empty in Polish spaces and we can see them as large sets. 

A {\emph{partially ordered set}} is a set $X$ with { a} relation $\leq$ {that} is reflexive, antisymmetric and transitive. Equivalently one may consider { a \emph{strict partial order}} $<$, that is a reflexive and transitive relation. If $\leq$ is partial order, then the corresponding strict partial order is given by $a<b\iff (a\leq b$ and $a\neq b)$. {\color{black}Similarly}, if $<$ is a strict partial order, then the corresponding partial order $\leq$ is given by $a\leq b\iff a<b$ or $a=b$. For $a,b\in X$, we write $a\perp b$ whenever $a$ is not comparable with $b$, that is neither $a\leq b$ nor $b\leq a$.

If $f$ is a function, then by $\dom f$ and $\img f$ we denote the domain and range of ${f}$, {\color{black}respectively}. We identify functions with their graphs. Therefore if $f,g$ are functions, then $g\subset f$ means that $f$ is an extension of $g$.  Clearly, a union $f\cup g$ of two functions is a function  iff $f$ and $g$ are equal on the common part $\dom f\cap\dom g$ of their domains. In particular $f\cup g$ {\color{black}is function} if $\dom f\cap\dom g=\emptyset$. Moreover, if $f,g$ are one-to-one functions and $\dom f\cap\dom g=\emptyset=\img f\cap\img g$, then $f\cup g$ is one-to-one. 

For a family $\mathcal{X}\subseteq\Aut(A)$ put {\color{black}$$\mathcal{X}^{<\omega}:=\{f\in\Ism(A):f\subset\bar{f}\mbox{ for some }\bar{f}\in\mathcal{X}\}.$$}  Clearly, $\mathcal{X}^{<\omega}=\{f_{\vert X}:X\mbox{ is a finite substructure of }A\}$, where $f_{\vert X}$ is the restriction of $f$ to $X$. Note that  if $A$ is ultrahomogeneous, then $\Ism(A)=\Aut(A)^{<\omega}$.

On a subset $\mathcal{X}\subseteq A$, we consider the topology induced from $\Aut(A)$.     Its basis consists of sets $\{\bar{f}\in \mathcal{X}:f\subset \bar{f}\}$, $f\in\mathcal{X}^{<\omega}$. Similarly, if $\mX\subset \Aut(A)\times \Aut(A)$, then we set $\mX^{<\omega}:=\{(f,g)\in\Ism(A)\times\Ism(A):f\subset \tilde{f}\;\mbox{and}\;g\subset\tilde{g}\;\mbox{for some }(\tilde{f},\tilde{g})\in\mX\}$ and the basis of the topology on $\mX$ consists of sets $\{(\bar{f},\bar{g})\in \mathcal{X}:f\subset \bar{f},\;g\subset \bar{g}\}$, $(f,g)\in\mathcal{X}^{<\omega}$. 

If $f,g$ are functions and $w(a,b)$ is a word, then by $w(f,g)$ we denote the function whose domain consist of points for which the appropriate compositions have sense. Moreover, if we write $"w(f,g)(x)"$, then we automatically assume that $x$ is in the domain of $w(f,g)$.

Below we prove two general theorems that we use further for particular cases.
 
\begin{theorem}\label{FirstGeneralTheorem}
Let $A$ be a countable ultrahomogeneous structure and let $\mathcal{X}\subseteq\Aut(A)\times \Aut(A)$. Assume that for every $(f_0,g_0)\in\mathcal{X}^{<\omega}$, every $h\in\Ism(A)$ and every irreducible word $w(a,b)$ there are $(f_1,g_1)\in\mathcal{X}^{<\omega}$ such that\begin{enumerate}
\item $f_0\subset f_1$ and $g_0\subset g_1$;
\item there is a word $\bar{w}(a,b)$ such that $h\subset\bar{w}(f_1,g_1)$;
\item $w(f_1,g_1)(y)\neq y$ for some $y\in A$.
\end{enumerate}
Then the set
\[
\{(f,g)\in\mathcal{X}:f\text{ and }g\text{ freely generate a dense subgroup of }\Aut(A)\}
\]
is comeager in $\mathcal{X}$. 
\end{theorem}

\begin{proof} For any word $w(a,b)$ and $h\in\Ism(A)$, define
\[
U_{w,h}:=\{(f,g)\in\mathcal{X}: w(f,g)\neq\on{id}\text{ and }h\subset\bar{w}(f,g)\text{ for some word }\bar{w}\}.
\]
We show that for any $(f_0,g_0)\in\mathcal{X}^{<\omega}$, the intersection $$U_{w,h}\cap \{(f,g)\in\mathcal{X}:f_0\subset f,\;g_0\subset g\}$$ contains a nonempty and open subset of $\mathcal{X}$. Let $(f_0,g_0)\in\mathcal{X}^{<\omega}$. By the assumption there are $(f_1,g_1)\in\mathcal{X}^{<\omega}$ fulfilling (i)--(iii). Therefore the set
\[
\{(f,g)\in\mathcal{X}:f_1\subset f\text{ and }g_1\subset g\}
\]
is contained in 
$$
U_{w,h}\cap \{(f,g)\in\mathcal{X}:f_0\subset f,\;g_0\subset g\}.
$$
Hence each $U_{w,h}$ contains an open and dense subset and, in particular,
$$
\bigcap_{w}\bigcap_{h}U_{w,h}=\{(f,g)\in\mathcal{X}:f\text{ and }g\text{ freely generate a dense subgroup of }\Aut(A)\}
$$
 is comeager in $\mathcal{X}$.
\end{proof}

In the following, if $f$ is a one-to-one function, $X$ is a set, and $k\in\Z$, then writing $f^k(X)$ we automatically assume that for every $x\in X$, the composition $f^k(x)$ is well defined, i.e., $X$ is subset of the natural domain of $f^k$.
\begin{theorem}\label{SecondGeneralTheorem}
Let $A$ be a countable ultrahomogeneous structure and let $\mathcal{X}$ be a $G_\delta$ subset of $\Aut(A)$. Assume that for every $f_0\in\mathcal{X}^{<\omega}$ and  every nonempty finite set $X\subset A$ there are $f_1\in\mathcal{X}^{<\omega}$ and $k\in\Z$ such that
\begin{enumerate}
\item $f_0\subset f_1$;
\item for any two $u,v\in\Ism(A)$ such that $\dom u\cup\img u\subset X$ and $\dom v\cup\img v\subset f_1^k(X)$ the union $u\cup v$ belongs to $\Ism(A)$. 
\end{enumerate}
Then for every $m\in\N$, the set of cyclically dense elements $\bar{g}\in\Aut(A)^m$ for the diagonal action is comeager in $\mathcal{X}\times \Aut(A)^m$.\\
Assume additionally that for any $(f_0,g_0)\in\mX^{<\omega}\times \Ism(A)$ and an irreducible word $w(a,b)$ there is $(f_1,g_1)\in\mX^{<\omega}\times \Ism(A)$ which extend $f_0$ and $g_0$ and such that ${w}(f_1,g_1)(y)\neq y$ for some $y$. Then the set
$$
\{(f,g)\in\mX\times\Aut(A):\mbox{$f$ and $g$ freely generates a dense subgroup of }\Aut(A)\}
$$
is comeager in $\mX\times\Aut(A)$.
\end{theorem}

\begin{proof}
For every $\overline{h}\in\Ism(A)^m$, define
\[
U_{\bar{h}}:=\{(f,\bar{g})\in\mathcal{X}\times\Aut(A)^m: \overline{h}\subseteq f^{-k}\bar{g}f^k\mbox{ for some }k\in\Z\}.
\]
We show that for every $f_0\in\mathcal{X}^{<\omega}$ and $\bar{u}\in\Ism(A)^m$, the intersection $$U_{\bar{h}}\cap \{(f,\bar{g})\in\mathcal{X}\times \Aut(A)^m:f_0\subset f,\;\bar{u}\subseteq \bar{g}\}$$ contains a nonempty and open set.

Let $f_0\in\mathcal{X}^{<\omega}$ and $\bar{u}\in\Ism(A)^m$. By the assumption,  for the set $X:=\bigcup_{i\leq m}(\dom u_i\cup\img u_i\cup\dom h_i\cup \img h_i)$, there is $f_1\in\mathcal{X}^{<\omega}$ and $k\in\Z$  fulfilling (i) and (ii). Define $v_i:=f_1^{k}h_if_1^{-k}:f_1^k(\dom h_i)\to f_1^{k}(\img h_i)$ for $i\leq m$. By (ii) we have that $u_i\cup v_i\in\Ism(A)$. Therefore 
\[
\{(f,\bar{g})\in\mathcal{X}\times\Aut(A)^m:f_{1}\subseteq f\text{ and }u_i\cup v_i\subset g_i,\;i\leq m\}
\]
is contained in 
\[
U_{\bar{h}}\cap \{(f,\bar{g})\in\mathcal{X}\times \Aut(A)^m:f_0\subset f,\;\bar{u}\subseteq \bar{g}\}
\].
This shows that $U_{\bar{h}}$ contains dense open subset of $\mathcal{X}\times\Aut(A)^m$. Hence
\[
\bigcap_{\bar{h}\in\Ism(A)^m}U_{\bar{h}}=\{(f,g)\in\mathcal{X}\times\Aut(A)^m:\{(f^kg_1f^{-k},\dots, f^kg_mf^{-k}):k\in\Z\}\text{ is dense in }\Aut(A)\}
\]
is a comeager $G_\delta$ subset of $\mathcal{X}\times\Aut(A)^m$. Since $\mathcal{X}$ is a Polish space, by the Kuratowski-Ulam theorem \cite[Theorem 8.41]{kechris} there is $f\in\mathcal{X}$ such that the set 
\[
\{g\in\Aut(A)^m:\{(f^kg_1f^{-k},\dots, f^kg_mf^{-k}):k\in\Z\}\text{ is dense in }\Aut(A)\}
\]
is comeager and it consists entirely of cyclically dense elements for the diagonal action. 

Now assume the additional part. We show that the assumptions of Theorem \ref{FirstGeneralTheorem} are satisfied for the family $\mX\times\Aut(A)$. Clearly $(\mX\times\Aut(A))^{<\omega}=\mX^{<\omega}\times \Ism(A)$.
Now take $(f_0,g_0)\in\mX^{<\omega}\times\Ism(A)$, $h\in\Ism(A)$ and an irreducible word $w$. Similarly as in the first part of the proof (for $m=1$), we can show that there exist an extension $(\hat{f},\hat{g})\in\mX^{<\omega}\times\Ism(A)$ of $(f_0,g_0)$, and $k\in\Z$ such that $h\subset \hat{f}^{-k}\hat{g}\hat{f}^k$. Now by additional assumption, there exists extension $(f_1,g_1)\in\mX^{<\omega}\times \Ism(A)$ such that ${w}(f_1,g_1)(y)\neq y$ for some $y$. Hence $(f_1,g_1)$ satisfies (i)-(iii) of Theorem \ref{FirstGeneralTheorem}.
\end{proof}

The hardest part of our work is to show that the assertions of Theorem \ref{FirstGeneralTheorem} and Theorem \ref{SecondGeneralTheorem} are fulfilled. In next sections we prove that we are always able to extend partial isomorphisms in an appropriate way.

%%%%%%%%%%%%%%%%%%%%%%%%%%%%%%%%%%%%%%%%%%%%%%%%%%%%%%%%%%%%%%%%%%%%%%%%%%%%%%
%%
%% B_n finite antichain of chains
%%
%%%%%%%%%%%%%%%%%%%%%%%%%%%%%%%%%%%%%%%%%%%%%%%%%%%%%%%%%%%%%%%%%%%%%%%%%%%%%%

\section{$\B_n$ -- the finite antichain of chains}\label{SectionBn}
In this section we deal with $\B_n$ {\color{black}for $n\in\N$}. Recall that {\color{black}$\B_n=(\{1,2,\dots,n\}\times\Q,\leq)$, where $\leq $ is} defined by
$$
(k,p)\leq (l,q)\iff k=l\text{ and }p\leq q.
$$
{\color{black}Later we will sometimes identify $\B_n$ with the underlying set $\{1,2,\dots,n\}\times\Q$.}\\
Let $\pi_1(k,p)=k$ and $\pi_2(k,p)=p$ be projections on the first and the second coordinate, respectively. If $A\subset\B_n$ and $k=1,...,n$, then the $k$-th section of $A$ is denoted by $A^k:=\{p\in\Q:(k,p)\in A\}$. We say that  $f\in \Ism(\B_n)$ is \emph{positive}, if $\pi_2(f(k,p))>p$ for all $(k,p)\in \on{dom}f$. By {$\Ism_+(\B_n)$ we denote the family of all positive partial automorphisms of $\B_n$. Similarly we define $\Ism_+(\Q)$, the family of all positive partial isomorphism of $\Q$ (in fact, $\Q$ and $\B_1$ can be identified).}

We say that $A>B$, where $A,B\subset\R$, if for every $a\in A$ and $b\in B$, $a>b$. If $M\in\R$, then $A>M$ means $A>\{M\}$. {Moreover, if $A,B\subset \{1,...,n\}\times \R$, then $A>B$ means that $\pi_2(A)>\pi_2(B)$.}

The following shows us what is the form of automorphisms of $\B_n$ {(by $S_n$ we denote the set of all permutations of $\{1,...,n\}$)}. 
\begin{proposition}\label{pp1}~ 
{\color{black}
\begin{enumerate}
\item Let $f:\B_n\to\B_n$. Then $f\in\Aut(\B_n)$ iff there exist $f_1,\dots,f_n\in\Aut(\Q)$ and $\tau_f\in S_n$ such that $f(k,p)=(\tau_f(k),f_k(p))$ for every $(k,p)\in\B_n$.\\
\item Let $f:X\to \B_n$ for some finite set $X\subset\B_n$. Then $f\in\Ism(\B_n)$ iff there exist $f_1,\dots,f_n\in\Ism(\Q)$ and $\tau_f\in S_n$ such that $f(k,p)=(\tau_f(k),f_k(p))$ for every $(k,p)\in X $. Moreover, $f\in\Ism_+(\B_n)$ iff each $f_i\in\Ism_+(\Q)$. 
\end{enumerate}}
\end{proposition}

\begin{proof} {\color{black}We first prove {(i)}}.
Assume that $f\in\on{Aut}(B_n)$. Then for every $k,l$ and $p,q$,
\begin{equation}\label{f1}
k=l \text{ and }p{\color{black}\leq}q\iff f(k,p){\color{black}\leq}f(l,q)\iff\pi_1(f(k,p))=\pi_1(f(l,q))\text{ and }\pi_2(f(k,p)){\color{black}\leq}\pi_2(f(l,q)).
\end{equation}
For every $k$, define $\tau_f(k):=\pi_1(f(k,p))$ for some $p$. By (\ref{f1}), the function $\tau_f$ is well defined (i.e., its value does not depend on the choice of $p$), and is one-to-one, and hence $\tau_f\in S_n$. For every $k$ and every $p$, put $f_k(p):=\pi_2(f(k,p)).$ Again by (\ref{f1}), $f_k$ is order preserving and consequently $f_k$ is one-to-one. We show that it is also onto. Take any $p'$ and let $p$ be such that $f(k,p)=(\tau_f(k),p')$. Then $f_k(p)=\pi_2(f(k,p))=p'$. 
Hence $f_k\in\on{Aut}(\Q)$. %\\
{\color{black}Now if $f(k,p)=(\tau_f(k),f_k(p))$ for some $\tau_f\in S_n$ and $f_1,...,f_n\in\Aut(\Q)$, then it is routine to check that $f\in\Aut(\Q)$.

Now we prove {(ii)}.
If $f\in\Ism(\B_n)$, then we can extend it to $\tilde{f}\in\Aut(\B_n)$, and find $\tau_{\tilde{f}}$ and $\tilde{f}_1,...,\tilde{f}_n$ as in (1.). Then the restrictions $f_k:=\tilde{f}_k\vert_{X^k}$ belongs to $\Ism(\Q)$, and, clearly, $f(k,p)=(\tau_{\tilde{f}}(k),f_k(p))$ for all $(k,p)\in X$. The opposite implication is obvious, as well as the last part of the statement.}
\end{proof}

{\color{black}
\begin{remark}\label{fr1}\emph{
In the case when $\pi_1(\on{dom}f)$ is proper subset of $\{1,...,n\}$, a permutation $\tau_f$ may not be uniquely determined - $\tau_f=\tau_{f'}$ for some extension $f'\in\Aut(\B_n)$. Hence, unless otherwise stated, by $\tau_f$ we consider any such permutation.}
\end{remark}}

%Obviously, each $f\in\Ism(\B_n)$ (positive) can be extended to some $\tilde{f}\in \Ism(\B_n)$ (positive) with all sections $(\on{dom}\tilde{f})^k$ nonempty (or $\pi_1(\dom\tilde{f})=\{1,2,\dots,n\}$). Thus we will usually consider such partial isomorphisms. In such a case by $\tau_f$ we will denote $\tau_{\bar{f}}$ where $\bar{f}\in\Aut(\B_n)$ is any automorphism extending $f$ ($\tau_{\bar{f}}$ the permutation adjusted to $\bar{f}$ in the sense of Proposition \ref{pp1}.}
%\end{remark}

The first lemma shows that we can always extend any positive {\color{black}$f\in\Ism(\B_n)$ by} adding a given point to its domain or range in such a way that the extension is still positive. 
\begin{lemma}\label{fil4}
Let {\color{black}$f\in\Ism_+(\B_n)$, $M\in\R$} and $A\subset \Q$ be finite.
\begin{enumerate}
\item If $(k,x)\notin{\color{black}\on{dom}f}$, then there is $y>x$, $y\notin A\cup{\color{black}\pi_2(\on{dom}f)}$ such that {\color{black}$\tilde{f}:=f\cup\{((k,x),(\tau_f(k),y))\}$ is a partial isomorphism. In particular, $\tilde{f}\in\Ism_+(\B_n)$.}
\item If $(k,y)\notin\on{rng}f$, $y>M$ and {\color{black}$\on{dom}f>M$}, then there is $x\in(M,y)\setminus(A\cup\pi_2(\on{rng}f))$ such that {\color{black}$\tilde{f}:=f\cup\{((\tau^{-1}_f(k),x),(k,y))\}$ is a partial isomorphism. In particular, $\tilde{f}\in\Ism_+(\B_n)$.}
\end{enumerate}
\end{lemma}

\begin{proof}Let $f_k$ be as in Proposition \ref{pp1} (i.e., the $k$-th coefficient of $f$).\\ %Let $\on{dom}(f)=\{a_1,...,a_k\}$ be such that $a_1<...<a_k$, and $\on{rng}(f)=\{b_1,...,b_k\}$ such that $b_1<...<b
(i) Let{ $a=\max(\on{dom}f_k\cap(-\infty,x))$ and $b=\min(\on{dom}f_k\cap(x,\infty))$ (here and in the sequel $\max\emptyset=-\infty$, $\min\emptyset=\infty$, $f_k(-\infty)=f_k^{-1}(-\infty)=-\infty$ and $f_k(\infty)=f_k^{-1}(\infty)=\infty$). Clearly $f_k(a)<f_k(b)$ and ${\color{black}x<b\leq f_k(b)}$. Then we choose any rational $y$ between $\max\{f_k(a),x\}$ and $f_k(b)$ which is not in $A\cup\pi_2(\on{dom}f)$.

(ii) Let $a=\max(\on{rng}f_k\cap(-\infty,y)))$ and $b=\min(\on{rng}f_k\cap(y,\infty))$. Then we choose any rational $x$ between $f_k^{-1}(a)$ and $\min\{f_k^{-1}(b),y\}$ which is not in $A\cup\pi_2(\on{rng}f)$.}
\end{proof}
As immediate consequence, we get
\begin{lemma}\label{f45}
Let {\color{black}$f\in\Ism_+(\B_n)$} and $p_0\in\Q$. Then there is {\color{black} $\tilde{f}\in\Ism_+(\B_n)$} such that $(k,p_0)\in\on{dom}\tilde{f}\cap\on{rng}\tilde{f}$ {\color{black}for} $k=1,...,n$, and $f\subseteq\tilde{f}$.
\end{lemma}

The next lemma shows that we can extend any {\color{black} $f\in\Ism_+(\B_n)$} so that a given point can be moved as far as we want.
\begin{lemma}\label{f25}
Let {\color{black}$f\in\Ism_+(\B_n)$}, $x\in\B_n$, $m\in\N$ and $M\in\R$.
Then there exists {\color{black} $\tilde{f}\in\Ism_+(\B_n)$} {\color{black}with $f\subseteq\tilde{f}$ and}
\begin{enumerate}
\item $x\in\on{dom}\tilde{f}$;
\item there is $l\geq m$ such that $\tilde{f}^l(x)$ is well defined and $\pi_2(\tilde{f}^l(x))>M$. 
\end{enumerate}
\end{lemma}

\begin{proof}
By {\color{black}Lemma \ref{fil4}}, we can assume that each section $(\on{dom}f)^k\neq \emptyset$. Set $\sigma:=\tau_f$ and 
enumerate $\on{dom}f=\{(k,a^k_i):k\leq n, i=1,\dots,t_k\}$ and $\on{rng}f=\{(\sigma(k),b^k_i):k\leq n, i=1,\dots,t_k\}$ in the way that $f(k,a_i^k)=(\sigma(k),b_i^k)$, and $a_i^k<a^k_{i+1}$ and (consequently) $b^k_i<b^k_{i+1}$.
Assume first that $x=(j,r)\notin\dom f$. If $r>a^{j}_{t_{j}}$, then we stop the procedure.

Now assume that $r\in(a^j_{i-1},a^j_{i})$ for some $i_0=1,2,\dots,t_j$ where $a^j_0=b^j_0=-\infty$. Since $a^j_{i_0}<b^j_{i_0}$ (because $f$ is positive), there is $r_1\in \Q\setminus\pi_2(\on{dom}f\cup\on{rng}f)$ such that $\max\{a^j_{i_0},b^j_{i_0-1}\}<r_1<b^j_{i_0}$.  Define a one point extension $f_1$ of $f$ such that $f_1(x)=(\sigma(j),r_1)=:x_1$.
If $r_1>a^{\sigma(j)}_{t_{\sigma(j)}}$, then we stop the procedure.

Now assume $r_1\in(a^{\sigma(j)}_{i-1},a^{\sigma(j)}_{i})$, for some $i=1,2,\dots,t_{\sigma(j)}$. Since $a^{\sigma(j)}_i<b^{\sigma(j)}_i$, there is $r_2\in \Q\setminus\pi_2(\on{dom}f_1\cup\on{rng}f_1)$ such that $\max\{a^{\sigma(j)}_i,b^{\sigma(j)}_{i-1}\}<r_2<b^{\sigma(j)}_i$.  Define a one point extension $f_2$ of $f_1$ such that $f_2(x_1)=(\sigma^2(j),r_2)=:x_2$.
If $r_2>a^{\sigma^2(j)}_{t_{\sigma^2(j)}}$, then we stop the procedure. If it is not the case, then we proceed as earlier.

We claim that such procedure stops after finitely many steps (which means that for some $n_0\in\N$, $r_{n_0}>a^{\sigma^{n_0}(j)}_{t_{\sigma^{n_0}(j)}}$). Suppose that it is not the case.
Let $l\leq n$ be such that $\sigma^l(j)=j$ (clearly, such $l$ exists because $\sigma$ has a finite rank). Then $$x_l=f_l(x_{l-1})=(\sigma^l(j),r_l)=(j,r_l)$$
 and $r_l>r_{l-1}>...>r_1>a_{i_0}^j$ (the value $i_0$ is that chosen in the first step of the construction). Therefore $r_l\in(a_{i_1-1}^j,a_{i_1}^j)$ for some $i_1>i_0$. Similarly, after next $l$ steps, we get $r_{2l}>a_{i_2}^j$ for some $i_2>i_1$. Thus after at most $l\cdot t_j$ steps, the procedure must stop. This gives a contradiction. Hence let $n_0\in\N\cup\{0\}$ be such that (we additionally set $r_0:=r$ and $f_0:=f$)
$$
r_{n_0}>a^{\sigma^{n_0}(j)}_{t_{\sigma^{n_0}(j)}}.
$$
Then $x_{n_0}=(\sigma^{n_0}(j),r_{n_0})=f_{n_0}^{n_0}(x)$ and $f_{n_0}$ is a positive partial isomorphism. Now let $m_0>n_0\in\N$ be such that $m_0\geq m$, and choose $r_{n_0+1}<...<r_{m_0}$ such that additionally
\begin{equation}\label{eq:aaa1}%\przemek:zmiana była podwójna etykieta
r_{n_0+1}>\max\{M, r_{n_0}, \max\{{\color{black}\pi_2(\on{dom}f_{n_0}\cup\on{rng}f_{n_0})}\}\}.
\end{equation}
Finally, define $$s(\sigma^{n_0+i}(j),r_{n_0+i}):=(\sigma^{n_0+i+1}(j),r_{n_0+i+1})$$
for $i=0,...,m_0-n_0-1$. In other words, $s$ behaves according to the diagram:
$$
(\sigma^{n_0}(j),r_{n_0})\stackrel{s}{\to}(\sigma^{n_0+1}(j),r_{n_0+1})\stackrel{s}{\to}\dots\stackrel{s}{\to}(\sigma^{m_0}(j),r_{m_0}).
$$
Clearly, $\tilde{f}:=f_{n_0}\cup s$ satisfies the thesis (in particular, by (\ref{eq:aaa1}) it is well defined and is partial isomorphism).

Now assume that $x\in\on{dom}f$. Then we can procedure as above, but considering $\tilde{x}:=f^k(x)$ instead of $x$, where $k\in\N$ is such that $f^k(x)$ is well defined but does not belong to $\on{dom}f$.% for $i<j$ and $f^j(x)\notin\on{dom}(f)$.
\end{proof}

Now, combining {\color{black}Lemmas \ref{fil4} and \ref{f25}}, we get that we can move the whole finite set as far as we want.
\begin{corollary}\label{f25'}
Let {\color{black}$f\in\Ism_+(\B_n)$, $C\subset\B_n$ be finite, {\color{black}$m,t\in\N$} and $M\in\R$}.
Then there exists {\color{black}$\tilde{f}_+\in\Ism(\B_n)$}, $f\subseteq\tilde{f}$ such that
\begin{enumerate}
\item $C\subset\on{dom}\tilde{f}$;
\item there is $l\geq m$ such that {\color{black}$\tilde{f}^{l+t}(x)$} is well defined for every $x\in C$ and {\color{black}$\tilde{f}^i(C)>M$ for $i=l,...,l+t-1$}. 
\end{enumerate}
\end{corollary}
\begin{proof}
Let $C=\{x_1,...,x_k\}$. Using Lemma \ref{f25} $k$ times, we get an extension {\color{black}$\tilde{f}\in\Ism_+(\B_n)$} so that for each $i=1,...,k$, there is $l_i\geq m$ with $\pi_2(\tilde{f}^{l_i}(x_i))>M$. Now let $l:=\max\{l_1,...,l_k\}$. Using Lemma \ref{fil4}, we can extend (if needed) $\tilde{f}$ to {\color{black} $\overline{f}\in\Ism_+(\B_n)$} so that for each $i=1,...,k$, the points $\overline{f}^{l_i}(x_i),\;\overline{f}^{l_i+1}(x_i),...,{\color{black}\overline{f}^{l+t-1}}(x_i)\in\on{dom}\overline{f}$. {\color{black}Clearly, it also holds }$\pi_2(\overline{f}^{l}(x_i))>M$.
\end{proof}

Next we are going to prove that for every {\color{black}irreducible} word $w(a,b)$ we can define positive partial automorphisms $r,s$ with $w(s,r)\neq\on{id}$, where $\on{id}$ is the identity function. We use two auxiliary results. Since $\Q$ can be {\color{black}identified with $\B_1$}, Lemma \ref{fil4} automatically implies
\begin{corollary}\label{f4}
Let {\color{black}$f\in\Ism_+(\Q)$}, $M\in\R$ and $A\subset \Q$ be finite.
\begin{itemize}
\item[(i)] If $x\notin\on{dom}f$, then there is $y>x$, $y\notin A\cup\on{dom}f$ such that {\color{black}$f\cup\{(x,y)\}\in\Ism_+(\B_n)$}.
\item[(ii)] If $y\notin\on{rng}f$, $y>M$ and {\color{black}$\on{dom}f>M$}, then there is $x\in(M,y)\setminus(A\cup\on{rng}f)$ such that {\color{black}$f\cup\{(x,y)\}\in\Ism_+(\Q)$}.%  is a positive partial isomorphism.
\end{itemize}
\end{corollary}

\begin{lemma}\label{f5}Let $M\in\R$ and let $w(a,b)$ be {\color{black}an irreducible word}. Then there are {\color{black} $r,s\in \Ism_+(\Q)$} such that $\on{dom}r,\on{dom}s>M$ and for some $p>M$, {\color{black} $w(s,r)(p)\neq p$}.
 \end{lemma}
\begin{proof}
We prove the statement by induction, with the following additional requirements:
\begin{itemize}
\item[(i)] if $w(a,b)=b^{n_k}a^{m_k}...b^{n_1}a^{m_1}$ and $n_k<0$, then $w(s,r)(p)\notin \on{dom}s\cup\on{rng}s\cup \on{rng}r$;
\item[(ii)] if $w(a,b)=b^{n_k}a^{m_k}...b^{n_1}a^{m_1}$ and $n_k>0$, then $w(s,r)(p)\notin \on{dom}s\cup\on{rng}s\cup \on{dom}r$;
\item[(iii)] if $w(a,b)=a^{m_k}...b^{n_1}a^{m_1}$ and $m_k< 0$, then $w(s,r)(p)\notin \on{dom}r\cup\on{rng}r\cup\on{rng}s$;
\item[(iv)] if $w(a,b)=a^{m_k}...b^{n_1}a^{m_1}$ and $m_k> 0$, then $w(s,r)(p)\notin \on{dom}r\cup\on{rng}r\cup\on{dom}s$.
\end{itemize}
{\color{black}(in the above formulation, we allow $m_1=0$, and in this situation we understand that the word ends with $b^{n_1}$).}\\
Let us prove it by induction with respect to the length of the word.

If $w(a,b)=a$ or $w(a,b)=a^{-1}$ or $w(a,b)=b$ or $w(a,b)=b^{-1}$, then we set $s(M+1)=M+2$ and $r(M+3)=M+4$. Then, clearly, the thesis holds (for $p=M+1$, {\color{black}$p=M+2$}, $p=M+3$ and {\color{black}$p=M+4$}, respectively). Assume that for the word $w(a,b)= b^{n_k}a^{m_k}...b^{n_1}a^{m_1}$ we have desired functions $r,s$ and a point $p$. Assume first {\color{black}that  $n_k<0$} (so {\color{black}(i)} is satisfied).
Then $w(s,r)(p)\neq p$ and $w(s,r)(p)\notin \on{dom}s\cup\on{rng}s\cup\on{rng}r$.

If $w'(a,b)=b^{-1} w(a,b)$, then we use {\color{black}Corollary \ref{f4}(ii)} for the point $x:=w(s,r)(p)$, function $r$, value $M$ and a set $A=\on{dom}s\cup\on{rng}s\cup\{p\}${\color{black}, and} we get an one-point extension $\tilde{r}=r\cup\{(x',x)\}$. Then $w'(s,\tilde{r})(p)=x'\neq p$ and $w'(s,\tilde{r})(p)=x'\notin \on{dom}s\cup\on{rng}s\cup \on{rng}\tilde{r}$. Hence (i) is satisfied.

If $w''(a,b)=aw(a,b)$, then we use {\color{black}Corollary \ref{f4}(i)} for the point $x=w(s,r)(p)$, function $s${\color{black}, and} a set $A=\on{dom}r\cup\on{rng}r\cup\{p\}$, and we get an one-point extension $\tilde{s}=s\cup\{(x,x'')\}$. Then $w''(\tilde{s},r)(p)=x''\neq p$ and $w''(\tilde{s},r)(p)=x''\notin \on{dom}r\cup\on{rng}r\cup\on{dom}\tilde{s}$. Hence (iv) is satisfied.

If $w'''(a,b)=a^{-1}w(a,b)$, then we use {\color{black}Corollary \ref{f4}(ii)} for the point $x:=w(s,r)(p)$, function $s$, value $M$ and a set $A=\on{dom}r\cup\on{rng}r\cup\{p\}$, and we get an one-point extension $\tilde{r}=r\cup\{(x''',x)\}$. Then $w'''(s,\tilde{r})(p)=x'''\neq p$ and $w'''(s,\tilde{r})(p)=x'''\notin \on{dom}r\cup\on{rng}r\cup \on{rng}\tilde{s}$).

In similar manner we can deal with the rest cases - using {\color{black}Corollary \ref{f4}} we get appropriate one-point extensions of $r$ or $s$ which satsify appropriate conditions from $(i)-(iv)$ with the original $p$.
\end{proof}

\begin{lemma}\label{KillingWordsInBn}
Let $M\in\R$ and let $w(a,b)$ be a word, and let $\eta,\xi\in S_n$. Then there are {\color{black} $s,r\in \Ism_+(\B_n)$} such that
\begin{enumerate}
\item {\color{black}$\on{dom}s,\on{dom}r>M$};
\item $w(s,r)(x)\neq x$ for some $x\in\B_n$;
\item $\tau_s=\eta$ and $\tau_r=\xi$.
\end{enumerate}
\end{lemma}

\begin{proof}
Let $r,s$ be as in Lemma \ref{f5}. For every $(l,q)\in\{1,...,n\}\times \on{dom}s$, define
$$
\tilde{s}(l,q)=(\eta(l),s(q))
$$
and for every $(l,q)\in\{1,...,n\}\times \on{dom}r$, define
$$
\tilde{r}(l,q)=(\xi(l),r(q)).
$$
Clearly, $\tilde{s},\tilde{r}\in{\color{black}\Ism_+(\B_n)}$, and, taking $p$ so that $w(s,r)(p)\neq p$, we have
$$w(\tilde{s},\tilde{r})(1,p)=(w(\eta,\xi)(1),w(s,r)(p))\neq (1,p)$$
\end{proof}
%Finally, we intend to show that we can always extend partial isomorphisms so that their {\color{black}compatible with a certain word is not identity.
{\color{black} Now we introduce some} further denotations.%\\
We say that a word $w(a,b)$ has {\color{black}\emph{positive terms}}, if $w$ is of the form
\begin{equation}\label{word}
w(a,b)=a^{n_j}b^{n_{j-1}}...b^{n_2}a^{n_1}
\end{equation}
where $n_1,...,n_j\geq 0$ and only $n_1,n_j$ can {\color{black}be} equal to $0$. If additionally $\eta,\xi\in S_n$, then we say that a family $A_1,...,A_{m+1}\subseteq \B_n$ is \emph{adjusted to $w$, $\eta$ and $\xi$}, if $m=|w|$ and for every $k=1,...,n$, $A_1^k<...<A_{m+1}^k$ and $$|A_1^k|=|A_2^{\eta(k)}|=...=|A_{n_1+1}^{\eta^{n_1}(k)}|=|A_{n_1+2}^{\xi(\eta^{n_1}(k))}|=...=|A_{n_1+n_2}^{\xi^{n_2}(\eta^{n_1}(k))}|=...=|A^{w(\eta,\xi)(k)}_{m+1}|$$ 
(we assumed above that $n_1>0$; if $n_1=0$, then we start with $|A_1^k|=|A_2^{\xi(k)}|=...$).\\
Each {\color{black}family of sets adjusted to $w$, $\eta$ and $\xi$ generates} a natural positive partial isomorphisms $s,r$ such that {\color{black}$w(s,r)(A_1)=A_{m+1}$}, according to the shape of $w$. More precisely, for every $k=1,..,n$, $s$ and $r$ act according to the following diagram: (we assume $n_1,n_j>0$; in other case we should omit the first or the last parts of terms):{\color{black}
$$
\{k\}\times A^k_1\stackrel{s}{\to}\{\eta(k)\}\times A^{\eta(k)}_2\stackrel{s}{\to}\dots\stackrel{s}{\to}\{\eta^{n_1}(k)\}\times A^{\eta^{n_1}(k)}_{n_1+1}\stackrel{r}{\to}
$$
$$
\stackrel{r}{\to}\{\xi(\eta^{n_1}(k))\}\times A^{\xi(\eta^{n_1}(k))}_{n_1+2}\stackrel{r}{\to}\dots\stackrel{r}{\to}\{\xi^{n_2}(\eta^{n_1}(k))\}\times A^{\xi^{n_2}(\eta^{n_1}(k))}_{n_1+n_2+1}\stackrel{s}{\to}\{\eta(\xi^{n_2}(\eta^{n_1}(k)))\}\times A^{\eta(\xi^{n_2}(\eta^{n_1}(k)))}_{n_1+n_2+2}\stackrel{s}{\to}\dots
$$  
$$
\dots\stackrel{s}{\to}\{w(\eta,\xi)(k)\}\times A^{w(\eta,\xi)(k)}_{m+1}.
$$}
We call such $s$ and $r$ as \emph{canonical} positive partial isomorphisms. %\\
{\color{black}It is well known that there are permutations $\eta,\xi\in S_n$ which generate $S_n$, i.e., such that for every $\tau\in S_n$, there is $w(a,b)$ with $w(\eta,\xi)=\tau$.}
\begin{lemma}\label{KillingIsom}
Let {\color{black}$M,M'\in\R$}, $\eta,\xi\in S_n$ be generators of $S_n$, $\tau\in S_n$, and let $A,B\subset\B_n$ be finite sets such that {\color{black}$M'<A^k<B^{\tau(k)}$} and $|A^k|=|B^{\tau(k)}|$ for every $k=1,...,n$. Then there exist {\color{black}$s,r\in\Ism_+(\B_n)$} and a word $w(a,b)$ of the form $w(a,b)=b^{-1}w'(a,b)b$ such that
\begin{itemize}
\item[(a)] $w(s,r)(A)=B$;
\item[(b)] $\tau_s=\eta$, $\tau_r=\xi$ and $w(\eta,\xi)={\color{black}\tau}$;
\item[(c)] $r(A),r(B)>M$;
\item[(d)] {\color{black}$\on{dom}r\cup \on{rng}r>M'$} and $\on{dom}s\cup\on{rng}s>M$.
\end{itemize}
\end{lemma}
\begin{proof}
At first, choose $c$ and $d$ so that $\max\{M,\max\{\bigcup_{k=1}^n B^k\cup A^k\}\}<c<d$, and take $\tilde{A},\tilde{B}\subseteq\B_n$ such that $c<\tilde{A}^k<d<\tilde{B}^{k}$, and $|A^k|=|\tilde{A}^{\xi(k)}|$ and $|B^k|=|\tilde{B}^{\xi(k)}|$ for every $k=1,...,n$. Then initially define {\color{black}$r^0\in\Ism_+(B_n)$} so that $r^0(A)=\tilde{A}$ and $r^0(B)=\tilde{B}$ and $\tau_{r^0}=\xi$.

Now let $\hat{w}(a,b)$ be a word with positive terms so that $\hat{w}(\eta,\xi)={\color{black}\tau}\xi^{-1}$ (such a word exists because each element of $S_n$ has finite rank). Assume that $\hat{w}$ is of the form (\ref{word}) and $m=|\hat{w}|$. Now take $c_1,...,c_m\in\R$ with $\tilde{B}^k<c_1<c_2<...<c_m$ for every $k=1,...,n$, and choose a family $A_1,...,A_{m+1}$ adjusted to $\hat{w}$, $\eta$ and $\xi$ such that for every $k=1,...,n$
\begin{itemize}
\item $A^k_1=\tilde{A}^k$,
\item $c_i<A^k_{i+1}<c_{i+1}$ for $i=1,...,m-1$,
\item $c_m<A^k_{m+1}$,
\end{itemize}
and let $r',s'$ be the canonical positive partial isomorphism adjusted to this family. 
Next, let {\color{black}$l_1<...<l_t$} be an increasing enumeration of $\{l\leq m+1:A_l\subseteq\on{dom}r'\}$. Choose {\color{black}$t_0> t+1$} such that $\xi^{{\color{black}t_0}+1}=\on{id}$. 

Assume that {\color{black}$n_1>0$ (we deal with the case $n_1=0$ analogously: in the following we should require $A^k_{l_i}<B_i<A^k_{l_i+1}$)} and choose a family $B_1,...,B_{{\color{black}t_0}+1}$ adjusted to the word $u(a,b)=b^{{\color{black}t_0}}$ and $\xi$, such that additionally for every $k=1,...,n$
\begin{itemize}
\item $B^k_1=\tilde{B}^k$;
\item $A^k_{l_i}<B^k_{i+1}<A^k_{l_i+1}$ for every $i=1,...,{\color{black}t}$;
\item $A^k_{m+1}<B^k_{l_j+2}$.
\end{itemize}
Then let $r''$ be a positive partial isomorphism adjusted to this family.
Finally, choose $i_0\in\N$ so that $\eta^{i_0}=\on{id}$ and a family $D_1,...,D_{i_0+1}$ adjusted to $u'(a,b)=a^{i_0}$ and $\eta$ such that additionally for every $k=1,...,n$, $D^k_1=A^k_{m+1}$ and $D^k_{i_0+1}=B^k_{{\color{black}t_0}+1}$. Then let $s''$ be a positive partial isomorphism adjusted to this family.

Define $s:=s'\cup s''$ and $r:=r^0\cup r'\cup r''$. We show that $r$ is a partial isomorphism. Let $x,y\in\on{dom}r$ be such that {\color{black}$\pi_2(x)<\pi_2(y)$}. If $x,y\in\on{dom}r_0$ or $x,y\in\on{dom}r'$ or $x,y\in\on{dom}r''$, then we are done. If $x\in\on{dom}r_0$ and $y\in\on{dom}(r'\cup r'')$, then {\color{black}$\pi_2(r_0(x))<c_1<\pi_2(r'\cup r''(y))$}. If $x\in\on{dom}r'$ and $y\in\on{dom}r''$, then $x\in A_{l_i}$ and {\color{black}$y\in B_a$ for some $a\geq i+1$}; thus $r'(x)\in A_{l_i+1}<B_{{\color{black}a}+1}\ni r''(y)$. If $x\in\on{dom}r''$ and $y\in\on{dom}r'$, then $x\in B_{i}$ and {\color{black}$y\in A_{l_a}$ for some $a\geq i$}; thus $r''(x)\in B_{i+1}<A_{l_i+1}\leq A_{l_{\color{black}a}+1}\ni r'(y)$. It is routine to check that $s$ is a positive partial isomorphisms.

Define $w(a,b):=b^{-{\color{black}t_0}-1}a^{i_0}\hat{w}(a,b)b$. Note that $r(A)=r_0(A)=\tilde{A}$, $\hat{w}(s,r)(\tilde{A})=\hat{w}(s,r)(A_1)=A_{m+1}$, $s^{i_0}(A_{m+1})=B_{{\color{black}t_0}+1}$, $r^{-{\color{black}t_0}}(B_{{\color{black}t_0}+1})=B_1=\tilde{B}$ and $r^{-1}(\tilde{B})=B$. Thus $w(s,r)(A)=B$ and we get (a). See that {\color{black}$\tau_s=\eta,\tau_r=\xi$, and}
$
{\color{black}w(\eta,\xi)=\xi^{-t_0-1}\eta^{i_0}\hat{w}(\eta,\xi)\xi=\tau},
$
{\color{black}hence} we obtain (b). Conditions (c) and (d) easily follow from the construction.
\end{proof}

\begin{lemma}\label{LemmaConditionForBeingExtension}
Let $f,h\in\Ism(B_n)$ be such that $\on{dom}h\subset\on{dom}f$, {\color{black} $\pi_2(\on{dom}h)=\{1,...,n\}$}, $f(\on{dom}h)=\on{rng}h$ and $\tau_f=\tau_h$. Then $h\subseteq f$.
\end{lemma}
\begin{proof}
Fix $k\leq n$ and {\color{black}let} $(\on{dom}h)^k=\{q_1<\dots<q_l\}$. {\color{black}Fix} $p_i=\pi_2(h(k,q_i))$. Then $(\on{rng}h)^{\tau_h(k)}=\{p_1<\dots<p_l\}$. Note that $f(k,q_i)\in\{\tau_h(k)\}\times\{p_1,\dots,p_l\}$. Since $f$ is order-preserving, then $f(k,q_i)=(\tau_h(k),p_i)$. Hence $h\subseteq f$.
\end{proof}

%We are ready to give a proof of Theorem \ref{main1}
\begin{lemma}(Key Lemma)\label{KeyLemmaBn}
Let {\color{black}$f\in\Ism_+(\B_n)$,} $g,h\in\Ism(\B_n)$ {\color{black}and} $w(a,b)$ be a word. {\color{black}If $\tau_f,\tau_g$ are generators of $S_n$, }then there are {\color{black}$\tilde{f}\in\Ism_+(\B_n)$ and $\tilde{g}\in\Ism(\B_n)$ such that}
\begin{enumerate}
\item[(i)] $f\subset\tilde{f}$ and $g\subset\tilde{g}$ and  $\tau_{\tilde{f}}=\tau_f$, $\tau_{\tilde{g}}=\tau_g$;
\item[(ii)] $h\subset\bar{w}(\tilde{f},\tilde{g})$ for some word $\bar{w}$;
\item[(iii)] $w(\tilde{f},\tilde{g})(y)\neq y$ for some $y\in\B_n$.
\end{enumerate}
\end{lemma}
{\color{black}
\begin{remark}\label{filrem1}\emph{
From Remark \ref{fr1} we know that, in general, permutation $\tau_f$ may be not uniquely determined. In such cases, the sentence "If $\tau_f,\tau_g$ are generators of $S_n$" in the assumptions of the above result should be understood as follows: "If there are extensions ${f'},{g'}\in\Aut(\B_n)$ of $f$ and $g$, respectively, such that $\tau_f=\tau_{{f'}}$ and $\tau_g=\tau_{{g'}}$". The same issue appears in the end of this section.}
\end{remark}}
\begin{proof}By Lemma \ref{fil4}
we may assume that $\pi_1(\dom f)=\pi_1(\dom g)=\{1,2,\dots,n\}$. Let $\eta:=\tau_f$ and $\xi:=\tau_g$. Let $M_0\in\Q$ be such that  
$$
M_0>\on{dom} f\cup\on{rng} f\cup\on{dom}g\cup\on{rng}g\cup\dom h\cup\img h.
$$ 
By Corollary \ref{f25'}, there exists {\color{black}$f_0\in\Ism_+(\B_n)$} and {\color{black}$m\in\N$} such that {\color{black}$\eta^{m}=\on{id}$, $f\subset f_0$ and $\pi_2({\color{black}f_0^m}(\dom h))>M_0$}. % Extending $f_0$, if necessary, we may assume that $\eta^m=\on{id}$.  
Now let
$$
M_1>\on{dom}f_0\cup\on{rng}f_0.
$$
Again by Corollary \ref{f25'}, there is {\color{black} $f_1\in\Ism_+(\B_n)$} and $m'\in\N$ such that $\eta^{m'}=\on{id}$, $f_0\subset f_1$ and $\pi_2(f_1^{m'}(\img h))>M_1$. By apply{\color{black}ing} Lemma \ref{KillingIsom} for $A:=f_1^{m}(\dom h)$ and $B:=f_1^{m'}(\img h)$, $M_2>\on{dom}{\color{black}f_1\cup\on{rng}f_1}$, {\color{black}$M'=M_0$} and $\tau=\tau_{h}$, we get  {appropriate \color{black}$r,s\in\Ism_+(\B_n)$} and a word $w'$. By the choice of $A$ and $B$, it is clear that $f_2:=f_1\cup s$ and $g_2:= g\cup r$ {\color{black}belong to $\Ism_+(\B_n)$, and}
$$
f_2^{-m'}\circ w'(f_2,g_2)\circ f_2^{m}(\dom h)=\img h.
$$
{\color{black}Also,}
$$
\tau_{(f_2^{-m'}\circ w'(f_2,g_2)\circ f_2^{m})}=\tau_{f_2^{-m'}}\tau_{w'(f_2,g_2)}\tau_{f_2^{m}}=\eta^{m'}\tau_{w_2(f_2,g_2)}\eta^m=w(\eta,\xi)=\tau_{h}.
$$
Hence by Lemma \ref{LemmaConditionForBeingExtension} we {\color{black}obtain}
$$
h\subset f_2^{-m'}\circ w'(f_2,g_2)\circ f_2^{m}.
$$
Let {\color{black}$s',r'\in\Ism_+(\B_n)$} be as in Lemma \ref{KillingWordsInBn}, chosen for $\eta,\xi$, $w$ and $M_3>\on{dom}f_2\cup\on{dom}g_2\cup\on{rng}f_2\cup\on{rng}g_2$. Then $f_3:=f_2\cup s'{\color{black}\in\Ism_+(\B_n)}$ and {\color{black}$g_3:=g_2\cup r'\in\Ism(\B_n)$. Since $g_3$ is not assumed to be positive} and for some $y\in \B_n$, $w(f_3,g_3)(y)\neq y$, {\color{black} then $f_3$ and $g_3$ satisfy the assertion.}
\end{proof}

{\color{black}Now put $$\Aut_+(\B_n):=\{f\in\Aut(\B_n):\forall_{(k,p)\in\B_n}\;\pi_2(f(k,p))>p\}$$
 and observe that for $f\in\Aut(\B_n)$,}
\[
f\in\Aut_+(\B_n)\iff{\color{black}\forall p\in\Q\forall k\leq n\exists q>p\exists l\leq n} f(k,p)=(l,q).
\]
Thus $\Aut_+(\B_n)={\color{black}\bigcap_{p\in\Q}\bigcap_{k\leq n}\bigcup_{q>p}\bigcup_{l\leq n}}\{f\in\Aut(\B_n):f(k,p)=(l,q)\}$. Since $\{f\in\Aut(\B_n):f(k,p)=(l,q)\}$ is clopen {\color{black}(in fact, it is the set of all extensions of partial isomorphism $(k,p)\to (l,q)$), we have that} $\Aut_+(\B_n)$ is $G_\delta$ subset of $\Aut(\B_n)$.\\

{\color{black}It is also easy to see that the set $$\mathcal{G}:=\{(f,g)\in\Aut(\B_n)\times\Aut(\B_n):\tau_f,\tau_g\;\mbox{are generators of }S_n\}$$ is open. Indeed, if $(f,g)\in \mathcal{G}$, then setting $f':=f\vert_{\{1,...,n\}\times \{0\}}$ and $g':=g\vert_{\{1,...,n\}\times \{0\}}$, we have that $f',g'\in \Ism(\B_n)$ and $$(f,g)\in\{(\tilde{f},\tilde{g})\in\Aut(\B_n)\times \Aut(\B_n):f'\subseteq \tilde{f},\;g'\subseteq\tilde{g}\}\subseteq \mathcal{G}.$$
Hence the family
$$
\mX:=\{(f,g)\in\Aut_+(\B_n)\times\Aut(\B_n):\tau_f,\tau_g\;\mbox{are generators of }S_n\}=\mathcal{G}\cap(\Aut_+(\B_n)\times\Aut(\B_n))
$$
is $G_\delta$ in $\Aut(\B_n)\times\Aut(\B_n)$.\\
Now let us note that for any $f,g\in\Ism(\B_n)$,
$$
(f,g)\in \mX^{<\omega}\iff f\in\Ism_+(\B_n)\;\mbox{and}\;\tau_f,\tau_g\;\mbox{are generators of }S_n.
$$
The implication $\Rightarrow$ is clear. Assume that $f\in\Ism_+(\B_n)$, and $\tau_f$ and $\tau_g$ are generators of $S_n$. Then using Lemma \ref{fil4} and the back-and-forth argument, we can inductively define $\tilde{f}\in\Aut_+(\B_n)$ such that $f\subset\tilde{f}$ and $\tau_{\tilde{f}}=\tau_f$, and also take any $\tilde{g}\in\Aut(\B_n)$ with $g\subset\tilde{g}$ and $\tau_{\tilde{g}}=\tau_g$. Then $(\tilde{f},\tilde{g})\in\mX$.\\
Hence, using} Theorem \ref{FirstGeneralTheorem} and Lemma \ref{KeyLemmaBn} we obtain the following. 
\begin{corollary}
The set {\color{black}
\[
\{(f,g)\in\mX:f\text{ and }g\text{ freely generate a dense subgroup of }\Aut(A)\}
\]
is comeager in $\mX$. }
\end{corollary}

%%%%%%%%%%%%%%%%%%%%%%%%%%%%%%%%%%%%%%%%%%%%%%%%%%%%%%%%%%%%%%%%%%%%%%%%%%%%%%
%%
%% B_\omega infinite antichain of chains
%%
%%%%%%%%%%%%%%%%%%%%%%%%%%%%%%%%%%%%%%%%%%%%%%%%%%%%%%%%%%%%%%%%%%%%%%%%%%%%%%

\section{$B_\omega$ -- the infinite antichain of chains}\label{SectionBOmega}
{\color{black}In this section we deal with $\B_\omega$. Recall that {\color{black}$\B_\omega=(\N\times\Q,\leq)$, where $\leq$ is defined by}
$$
(k,p)\leq (l,q)\iff k=l\text{ and }p\leq q.
$$
{\color{black}Again, we will identify $\B_\omega$ with $\N\times\Q$.}\\
Symbols $\pi_1(\cdot),\pi_2(\cdot),A^k$ and so on have analogous meaning as in the previous section. The following result is a counterpart of Proposition \ref{pp1}. We skip the proof since it is essentially the same. $S_\omega$ denotes the family of all permutations of $\N$.
%{\color{black}Similarly as in previous section, one can show the following
\begin{proposition}\label{ppp1}~ {\color{black}
\begin{enumerate}
\item Let $f:\B_\omega\to\B_\omega$. Then $f\in\Aut(\B_\omega)$ iff there exist $f_1,f_2,\dots\in\Aut(\Q)$ and $\tau_f\in S_\omega$ such that $f(k,p)=(\tau_f(k),f_k(p))$ for every $(k,p)\in\B_\omega$.
\item Let $f:X\to Y$ for some finite sets $X,Y\subset\B_\omega$, and let $N_f:=\{k\in\N:X^k\neq\emptyset\}$. Then $f\in\Ism(\B_n)$ iff there exist $f_k\in\Ism(\Q)$, $k\in N_f$, and one to one map $\tau_f:N_f\to\N$ such that $f(k,p)=(\tau_f(k),f_k(p))$ for every $(k,p)\in X $. Moreover, $f\in\Ism_+(\B_n)$ iff each $f_k\in\Ism_+(\Q)$. 
\end{enumerate}}
\end{proposition}
}
 By $\mathcal{X}\subseteq\Aut(\B_\omega)$ we denote the set of all $f\in\Aut(\B_\omega)$ such that the corresponding permutation $\tau_f\in S_\infty$ does not have finite cycles, i.e. the set $\{\tau_f^k(n):k\in\omega\}$ is infinite for each $n\in\omega$.  By $\mathcal{X}_0\subseteq\Ism(\B_\omega)$ we denote the set of all $f\in\Ism(\B_\omega)$ such that the corresponding partial permutation $\tau_f$ does not have finite cycles, i.e. for every $n\in\dom\tau_f$ there is $k\in\omega$ such that $\tau_f^k(n)\notin\dom\tau_f$. {\color{black}It turns out that
\begin{lemma}\label{lem:aaa1}%\przemek zmiana etykiety
$\mathcal{X}_0=\mathcal{X}^{<\omega}$
\end{lemma}
To prove it, we need the following:}

\begin{lemma}\label{BOmegaOnePointExtension}
Let $f\in\mathcal{X}_0$ and $(n,p)\in\B_\omega$. There is $k\geq 1$ and an extension $\hat{f}\in\mathcal{X}_0$ of $f$ such that $\hat{f}^k(n,p)$ is defined and $\tau_{\hat{f}}^k(n)\notin\on{dom}\tau_f$. Moreover, {\color{black}if $n\notin\on{dom}\tau_f$, then we can take $k=1$.}% there is an extension $\hat{f}\in\mathcal{X}_0$ of $f$ such that $(n,p)\in\on{dom}\hat{f}$ and $\tau_{\hat{f}}(n)\notin\on{dom}\tau_f$ if .
\end{lemma}

\begin{proof}
If $n\notin\on{dom}\tau_f$, then find $m\notin\on{dom}\tau_f\cup\{n\}$ and define an extension {\color{black}$\hat{f}:=f\cup\{((n,p),(m,0))\}$}. Clearly $\hat{f}\in\mathcal{X}_0$. Note that $\tau_{\hat{f}}(n)=m\notin\dom\tau_f$ which gives us the {\color{black}"moreover part"} of the assertion. 

If $n\in\on{dom}\tau_f$, then find {\color{black}$l\in\N\cup\{0\}$} such that $\tau_f^l(n)\in\on{dom}\tau_f$ and $\tau_f^{l+1}(n)\notin\on{dom}\tau_f$. {\color{black}Take an extension $\tilde{f}\in \Aut(A)$ of $f$, set $X:=\{(n,p),\tilde{f}(n,p),...,\tilde{f}^l(n,p)\}$ and define ${f_1}:=\tilde{f}\vert_{\on{dom}f\cup X}$}. Then $\tau_f=\tau_{f_1}$ and $\tau_{f_1}^{l+1}(n)\notin\on{dom}\tau_{f_1}$. Thus $f_1\in\mathcal{X}_0$ and the point $(n',p')=f_1^{l+1}(n,p)$ has the property that $n'\notin\on{dom}\tau_{f_1}$. Proceeding as in the previous case we find a desired extension $\hat{f}\in\mathcal{X}_0$ of $f_1$. 
\end{proof}

We are ready to {\color{black}prove Lemma \ref{lem:aaa1}}. Clearly {\color{black}$\mathcal{X}^{<\omega}\subseteq\mathcal{X}_0$}. Let $f\in\mathcal{X}$. Note that $g\in\mathcal{X}_0$ if and only if $g^{-1}\in\mathcal{X}_0$. Therefore using Lemma \ref{BOmegaOnePointExtension} and the back-and-forth argument we can inductively define an extension $\tilde{f}\in\mathcal{X}$ of $f$. Thus $\mathcal{X}_0\subseteq\mathcal{X}^{<\omega}$.

\begin{lemma}\label{BOmegaDalekoWBok}
Let $f\in\mathcal{X}_0$ and $A\subset\B_\omega$ {\color{black}be finite}. There is $k\geq 1$ and an extension $\hat{f}\in\mathcal{X}_0$ of $f$ such that $\hat{f}^k(A)$ is defined and {\color{black}$\tau_{\hat{f}}^k(\pi_1(A))\cap A=\emptyset$.}
\end{lemma}

\begin{proof}{\color{black}Let $A=\{(n_i,p_i):i\leq j\}$. By Lemma \ref{lem:aaa1}, there exists an extension $\tilde{f}\in\mX$ of $f$. Since $\tau_{\tilde{f}}$ does not contain cycles, for every $i=1,...,j$, there is $k_i\in\N$ such that for $k\geq k_i$, $\tau_{\tilde{f}}^{k_i}(n_i)>\pi_1(A)$. Hence take $k_0:=\max\{k_1,...,k_j\}$, choose $X=A\cup\{\tilde{f}^l(A):l=1,...,k_0-1\}$ and set $\hat{f}:=\tilde{f}\vert_{X\cup\on{dom}f}$. Then $\hat{f}\in\mX_0$ and $\tau_{\hat{f}}^{k_0}(A)\cap \pi_1(A)=\emptyset$.}
%. Using Lemma \ref{BOmegaOnePointExtension} we find inductively $f\subset f_1\subset f_2\subset\dots\subset f_j$ and $k_i$ such that $f_i\in\mathcal{X}_0$ and $f^{k_i}(n_i,p_i)$ is defined and $\tau_{f_i}(n_i)\notin\on{dom}\tau_f$. Take $k=\max_{i\leq j}k_i$ and use Lemma \ref{BOmegaOnePointExtension} to find an extension $\hat{f}\in\mathcal{X}_0$ of $f_j$ such that $\hat{f}^k(n_i,p_i)$ is defined and $\tau_{\hat{f}}(A)\cap\on{dom}\tau_f=\emptyset$. 
\end{proof}

\begin{lemma}\label{UnionOfTwoPartialIsomorphisms}
Let $h_0,h_1\in\Ism(\B_\omega)$ be such that 
\[
\pi_1(\dom h_0\cup\img h_0)\cap\pi_1(\dom h_1\cup\img h_1)=\emptyset.
\]
Then $h_0\cup h_1\in\Ism(\B_\omega)$.
\end{lemma}

\begin{proof}
Note that $\dom h_0\cap \dom h_1=\emptyset=\img h_0\cap\img h_1$. As we have mentioned in Introduction, this implies that $h:=h_0\cup h_1$ is a one-to-one function. We need to show that $h$ is order-preserving. Let $(n,p),(k,q)\in\B_\omega$. If $(n,p),(k,q)\in\dom h_i$, then $(n,p)\leq(k,q)\iff h_i(n,p)\leq h_i(k,q){\color{black}\iff} h(n,p)\leq h(k,q)$. If $(n,p)\in\dom h_0$ and $(k,q)\in\dom h_1$, then $n\neq k$ which means that the points $(n,p)$ and $(k,q)$ are $\leq$-incomparable. Since the first coordinates of $h(n,p)=h_0(n,p)$ and $h(k,q)=h_1(k,q)$ are different, $h(n,p)$ and $h(k,q)$ are $\leq$-incomparable as well. 
\end{proof}

\begin{lemma}\label{WordLemma}
Let $w(a,b)$ be a word of the form {\color{black}$a^{m_k}b^{n_k}\dots a^{m_0}b^{n_0}$} where $n_0,m_k\in\Z$ and {\color{black}only $n_0,m_k$ can be $0$}. Let $Y=\{1,2\dots,{\color{black}M+1}\}$ where ${\color{black}M}=\sum_{i=0}^k({\color{black}|n_i|+|m_i|})$. Then there are $A,B\subseteq Y$ and $u:A\to Y$, $v:B\to Y$ such that $u,v$ have no cycles and $w(u,v)(1)={\color{black}M+1}$. 
\end{lemma}

\begin{proof} We proceed inductively with respect to $k$, with additional requirement that:
\begin{center}if $m_k\neq 0$, then $M+1\notin \on{dom}v\cup\on{rng}v$.\end{center}
Let $w=a^{m_0}b^{n_0}$. Consider cases:%Assume first that $m_0,n_0\neq 0$ and consider cases:\\
\begin{itemize}
\item $m_0,n_0>0$. Then we put $B_0:=\{1,...,n_0\},\;A_0:=\{n_0+1,...,n_0+m_0\}$ and define $v:B_0\to\{1,...,n_0+m_0+1\}$ and $u:A_0\to\{1,...,n_0+m_0+1\}$ by $v(a):=a+1$ and $u(a):=a+1$.
\item $m_0<0,\;n_0>0$. Then we put $B_0:=\{1,...,n_0\},\;A_0:=\{n_0+2,...,n_0+|m_0|+1\}$ and define $v:B_0\to\{1,...,n_0+m_0+1\}$ and $u:A_0\to\{1,...,n_0+|m_0|+1\}$ by $v(a):=a+1$ and $u(a):=a-1$.
\item $m_0>0,n_0<0$ or $m_0,n_0<1$ - we define $v,u$ is analogous way.
\item $m_0>0,\;n_0=0$. We set $B_0:=\emptyset$, $A_0:=\{1,...,m_0\}$, $v=\emptyset$, $u(a):=a+1$.
\item $m_0<0,\;n_0=0$ or $m_0=0,\;n_0>0$ or $m_0=0,\;n_0<0$ - we define $v,u$ in analogous way.
\end{itemize}
Then $w(1)=|n_0|+|m_0|+1$ and if $m_0\neq 0$, then $|n_0|+|m_0|+1\notin \on{dom}v\cup\on{rng}v$.

Assume that we have already defined $v,u$ for some word $w=a^{m_k}b^{n_k}\dots a^{m_0}b^{n_0}$ with $m_k\neq 0$, and consider a word $w'=a^{m_{k+1}}b^{n_{k+1}}w$, where $n_{k+1}\neq 0$. Set $M:=\sum_{i=0}^k(|m_i|+|n_i|)$. Similarly as in the initial step, we can define $\overline{A},\overline{B}\subseteq\{M+1,...,M+|n_{k+1}|+|m_{k+1}|+1\}$ and $\overline{v}:\overline{B}\to \{M+1,...,M+|n_{k+1}|+|m_{k+1}|+1\}$ and $\overline{u}:\overline{A}\to \{M+1,...,M+|n_{k+1}|+|m_{k+1}|+1\}$ so that $\overline{u}^{m_{k+1}}(\overline{v}^{n_{k+1}}(M+1))=M+|n_{k+1}|+|m_{k+1}|+1$. Then $u:=u'\cup\overline{u}$ and $v:=v'\cup\overline{v}$ satisfy required conditions.
\end{proof}

\begin{lemma}(Key Lemma)\label{KeyLemmaBOmega}
Let $f\in\mathcal{X}_0$ and $g\in\Ism(\B_\omega)$. Let $X$ be a finite subset of $\B_\omega$ and let $w(a,b)$ be {\color{black}an irreducible word}. Then there are a natural number $k\in{\color{black}\N}$ and partial isomorphisms $\tilde{f}\in\mathcal{X}_0$ and  $\tilde{g}\in\Ism(\B_\omega)$ such that
\begin{enumerate}
\item $f\subseteq\tilde{f}$, $g\subseteq\tilde{g}$;
\item for any $h_1,h_2\in\Ism(\B_\omega)$ such that ${\color{black}\dom h_1\cup\img h_1}\subseteq X$ and ${\color{black}\dom h_2\cup\img h_2}\subseteq\tilde{f}^k(X)$, the function {\color{black}$h_1\cup h_2\in\Ism(\B_\omega)$};
\item there is $y\in\B_\omega$ such that $w(\tilde{f},\tilde{g})(y)\neq y$. 
\end{enumerate}
\end{lemma}

\begin{proof}
Since $f$ and $g$ are partial isomorphisms, there is $N\in{\color{black}\N}$ such that $\pi_1(\dom f\cup \img f\cup \dom g\cup\img g)\subseteq\{0,1,\dots,N\}$. Let $u,v,A,B,Y$ be as in the assertion of Lemma \ref{WordLemma} for the word $w$ {\color{black}(which, clearly, can be written in the appropriate form)}. For a set $L\subset{\color{black}\N}$ by $L+N$ we understand the set $\{l+N:l\in L\}$. Define $u':A+N\to Y+N$ and $v':B+N\to Y+N$ by the formulas $u'(l+N)=u(l)+N$ and $v'(l+N)=v(l)+N$. Clearly {\color{black}$w(u',v')(N+1)=N+M+1$ where $M$ is as in Lemma \ref{WordLemma}}. {\color{black}Now for $l\in\on{dom}u'$, define $f_u(l,0):=(u'(l),0)$ and, similarly, for $l\in\on{dom}v'$, define $g_v(l,0):=(v'(l),0)$.} Let {\color{black}$f_1:=f\cup f_u$ and $g_1:=g\cup g_v$. Since $f_u,g_v\in\Ism(\B_\omega)$,} by Lemma \ref{UnionOfTwoPartialIsomorphisms} we obtain that $f_1,g_1\in\Ism(\B_\omega)$. Since $u'$ have no cycles, {\color{black}we proved that} $f_1\in\mathcal{X}_0$. 

By Lemma \ref{BOmegaDalekoWBok} there are $k\in{\color{black}\N}$ and an extension $\tilde{f}{\color{black}\in\mX_0}$ of $f_1$ such that $\tilde{f}^k(X)$ is defined and $\tau^k_{\tilde{f}}(\pi_1(X))$ is disjoint with ${\color{black}X}$.  By Lemma \ref{UnionOfTwoPartialIsomorphisms} we obtain ({\color{black}ii}). Put $\tilde{g}=g_1$ and observe that (i) and (iii) are fulfilled as well {\color{black}(as $w(\tilde{f},\tilde{g})(N+1,0)=(N+M+1,0)$).} 
\end{proof}

Let us observe that $\mathcal{X}$ is a $G_\delta$ subset of $\Aut(\B_\omega)$. Indeed, fix $(k_1,p_1),(k_2,p_2),\dots,(k_l,p_l)\in\B_\omega$ such that $k_1=k_l$. Then the set $\{f\in\Aut(\B_\omega):f(k_i,p_i)=(k_{i+1},p_{i+1})\text{ for }i<l\}$ is {\color{black}clopen} subset of $\Aut(\B_\omega)$. {\color{black}Clearly,} $f$ has a finite cycle if and only if there are $(k_1,p_1),(k_2,p_2),\dots,(k_l,p_l)\in\B_\omega$  such that $k_1=k_l$ and $f(k_i,p_i)=(k_{i+1},p_{i+1})$ for $i<l$. Therefore the set $\{f\in\Aut(\B_\omega):f \text{ has a finite cycle }\}$ is an $F_\sigma$ subset of $\Aut(\B_\omega)$. Thus $\mathcal{X}$ is $G_\delta$ in $\Aut(\B_\omega)$. 

Using Key Lemma and {\color{black}Theorem \ref{SecondGeneralTheorem}} we obtain the following.

\begin{corollary}~
\begin{enumerate}
\item The set
\[
\{(f,g)\in\mathcal{X}\times\Aut(\B_\omega):f\text{ and }g\text{ freely generate a dense subgroup of }\Aut(\B_\omega)\}
\]
is comeager in $\mathcal{X}\times\Aut(\B_\omega)$.
\item {\color{black}For every $m\in\N$, t}he set of cyclically dense elements $\bar{g}\in\Aut(\B_\omega)^m$ for the diagonal action in comeager in $\Aut(\B_\omega)^m$. 
\end{enumerate}
\end{corollary}

%%%%%%%%%%%%%%%%%%%%%%%%%%%%%%%%%%%%%%%%%%%%%%%%%%%%%%%%%%%%%%%%%%%%%%%%%%%%%%
%%
%% C_n and C_\omega chain of antichains
%%
%%%%%%%%%%%%%%%%%%%%%%%%%%%%%%%%%%%%%%%%%%%%%%%%%%%%%%%%%%%%%%%%%%%%%%%%%%%%%%

\section{$\mC_n$ -- the chain of antichains}\label{SectionC}

{\color{black}Let $n\leq\omega$. Recall that by $\mC_n$ we mean the partially ordered set $(\{1,...,n\}\times\Q,\leq)$, provided $n<\omega$, and $(\N\times\Q,\leq)$, if $n=\omega$, where $\leq$ is given by $(k,p)\leq (l,q)\iff p\leq q$.\\
Again, we will identify $\mC_n$ with the underlying set.}\\
We say that $F\in \Ism(\mC_n)$ is \emph{positive} if for every $(k,p)\in\on{dom}F$, $\pi_2(F(k,p))>p$. The family of all positive partial isomorphisms is denoted by $\Ism_+(\mC_n)$.  
{\color{black}The following result is a counterpart of Propositions \ref{pp1} and \ref{ppp1}. If $X\subset \mC_n$ and $p\in\Q$, then we set $X_p:=\{k\in\omega:(k,p)\in X\}$. %Also, let $S_n^{<\omega}$ denote the set of all partial permutations of $S_n$, i.e., all one-to-one maps $\tau:A\to\{k\leq\omega:k\leq n\}$, where $A\subset\omega$ is finite.

\begin{proposition}~
\begin{enumerate}
\item Let $F:\mC_n\to\mC_n$. $F\in\on{Aut}(\mC_n)$ iff there exist $f_F\in\on{Aut}(\Q)$ and $\tau_{F,p}\in S_n$, $p\in\Q$, such that $F(k,p)=(\tau_{F,p}(k),f_F(p))$ for every $(k,p)\in\mC_n$.
\item Let $F:X\to \mC_n$ for some finite set $X\subset \mC_n$, and let $N_F:=\{p\in\Q:X_p\neq\emptyset\}$. Then $F\in \Ism(\mC_n)$ iff there exist $f_F\in\Ism(\Q)$ with $\dom f_F=N_F$ and one-to-one maps $\tau_{F,p}:X_p\to\{k\in\omega:k\leq n\}$, $p\in N_F$, such that $F(k,p)=(\tau_{F,p}(k),f_F(p))$ for every $(k,p)\in X$. Additionally, $F\in\Ism_+(\mC_n)$ iff $f_F\in\Ism_+(\Q)$.
\end{enumerate}
\end{proposition}}
\begin{proof}{\color{black}We first prove (i). Let $F\in\Aut(\mC_n)$. For $p\in\Q$, set $f_F(p):=\pi_2(F(k,p))$ for some $k\in\N$. The map $f_F$ is well defined since for $k,l\in\N$, $\pi_2(F(k,p))=\pi_2(F(l,p))$ (as $F$ is partial isomorphism)}. Take any rational numbers $p,q$ with $p{\color{black}\leq}q$. Then $(k,p){\color{black}\leq}(m,q)$, and consequently $F(k,p){\color{black}\leq}F(m,q)$. Thus $f(p){\color{black}\leq}f(q)$ which means that $f$ is an authomorphism of $(\Q,{\color{black}\leq})$. {\color{black}By the above observations, for every $p\in\Q$ we also have} $F(\{k\in\omega:k\leq n\}\times\{p\})=\{k\in\omega:k\leq n\}\times\{f(p)\}${\color{black}. Since $F$ is a bijection, then $\tau_{F,p}$ defined by $\tau_{F,p}(k)=\pi_1(F(k,p))$} is in $S_n$.\\{\color{black}Now if $F(k,p)=(\tau_{F,p}(k),f_F(p))$ for some permutations $\tau_{F,p}$ and $f_F\in\Aut(\Q)$, then it is routine to check that $F\in\Aut(\mC_n)$.\\
Now we show (ii). If $F\in\Ism(\mC_n)$, then we can find its extension $\tilde{F}\in\Aut(\mC_n)$. Then $f_F$ and appropriate $\tau_{F,p}$ are restrictions of $f_{\tilde{F}}$ and $\tau_{\tilde{F},p}$.\\
The opposite implication is obvious, as well as the last part of the statement.}
\end{proof}{\color{black}
\begin{remark}\emph{Let us remark that in the case $n<\infty$ and $F\in\Ism(\mC_n)$, the partial permutations $\tau_{F,p}$ may belong to $S_n$.}
\end{remark}}

\begin{lemma}\label{CRozszerzanieOJedenPunkt}
Let {\color{black}$F\in\Ism_+(\mC_n)$} and $(k,p)\in\mC_n$.
\begin{enumerate}
\item There is a positive extension $\bar{F}$ of $F$ such that $(k,p)\in\on{dom}\bar{F}$. 
\item There is a positive extension $\bar{F}$ of $F$ such that $(k,p)\in\on{rng}\bar{F}$.
\end{enumerate}
\end{lemma}

\begin{proof}
(i) {\color{black}Let} $\pi_2(\on{dom}F)=\{p_1<p_2<\dots<p_{\color{black}m}\}$. First assume that $p=p_i$ for some $i$ {\color{black}and let $\tau\in S_n$ be any extension of $\tau_{F,p}$}. If {\color{black}$(k,p)\in\on{dom}F$}, then put $\bar{F}:=F$. {\color{black} If $(k,p)\notin\on{dom}F$, then we set $\bar{F}:=F\cup\{\big((k,p),(\tau(k),f_F(p))\big)\}$.}

Now, assume that $p\notin\pi_2(\on{dom}F)$. Then there is $i=0,1,\dots,m$ such that $p_i<p<p_{i+1}$ where $p_0=-\infty$ and $p_{m+1}=\infty$. Since $f$ is positive, $\max\{p,f(p_i)\}<f(p_{i+1})$ {\color{black}where $f(-\infty):=-\infty$ and $f(\infty):=\infty$}). Take any rational number $q$ from $(\max\{p,f(p_i)\},f(p_{i+1}))$ {\color{black}and put} $\bar{F}:=F\cup\{\big((k,p),(1,q)\big)\}$.

The second part can be proved in a similar way. 
\end{proof}

\begin{lemma}\label{CJedenPunktWysokoDoGory}
Let $M\in{\color{black}\R}$, $(l,p)\in\mC_n$ and {\color{black}$F\in\Ism_+(\mC_n)$}. Then there are $k\in\N$ and a positive extension $\bar{F}$ of $F$ such that $\pi_2(\bar{F}^k(l,p))>M$.
\end{lemma}

\begin{proof}
{\color{black}Let} $\pi_2(\on{dom} F)=\{p_1<p_2<\dots<p_m\}$. If $p>p_m$, then take any $q$ greater than $\max\{p,f(p_m),M\}$ and put $\bar{F}:=F\cup\{((l,p),(1,q))\}$. Then $\bar{F}$ is positive {\color{black}and $\pi_2(\bar{F}(l,p))>M$.}

If $p=p_i$ for some $i{\color{black}=1,...,m}$, then using Lemma \ref{CRozszerzanieOJedenPunkt} we find a positive extension $F'$ of $F$ such that $(l,p)\in\on{dom}F'$. By positivity of $F'$ we have $\pi_2(F'(l,p))>p_i$. 

If $p_{i-1}<p<p_i$ ($p_0=-\infty$ and $f(-\infty)=-\infty$), then find $q$ such that $\max\{p_i,f(p_{i-1})\}<q<f(p_i)$ and put ${\color{black}{F'}:}=F\cup\{((l,p),(1,q))\}$. 

We have shown that if $p_{i-1}<p\leq p_i$, then there is a positive extension $F'$ of $F$ such that $p_i<\pi_2(F'(l,p))$. {\color{black}If $\pi_2(F'(l,p))>p_m$, then we stop the procedure. Otherwise} $p_{j-1}<\pi_2(F'(l,p))\leq p_j$ for some ${\color{black}m\geq }j>i$. In the next step we extend $F'$ to a positive $F''$ with $p_j<\pi_2(F''(F''(l,p)))$. After finitely many (say $k$ many) steps we find a positive extension {\color{black}$\tilde{F}$ with $\pi_2(\tilde{F}^k(l,p))>p_m$. Then we extend it to $\bar{F}$ so that $\pi_2(\bar{F}^{k+1}(l,p))>M$ (as in the first part of the proof)}.
\end{proof}

\begin{lemma}\label{CZbiorXWysokoDoGory}
Let $M\in{\color{black}\R}$, $X\subset\mC_n$ be finite and {\color{black}$F\in\Ism_+(\mC_n)$}. Then there are $k\in\N$ and a positive extension $\bar{F}$ of $F$ such that $\pi_2(\bar{F}^k(X))>M$.
\end{lemma}

\begin{proof}
Let $X=\{(l_i,p_i):i=1,\dots,m\}$. Using Lemma \ref{CJedenPunktWysokoDoGory} we find a positive extension $F_1$ of $F$ and positive integer $k_1$ such that $\pi_2(F_1^{k_1}(l_1,p_1))>M$. Then using again Lemma \ref{CJedenPunktWysokoDoGory} we find a positive extension $F_2$ of $F_1$ and positive integer $k_2$ such that $\pi_2(F_2^{k_2}(l_2,p_2))>M$. Proceeding inductively we {\color{black}find} $F_1\subset F_2\subset\dots\subset F_m$ such that $\pi_2(F_m^{k_i}(l_i,p_i))>M$ for $i\leq m$. Let $k=\max_{i\leq m}k_i$. Using Lemma \ref{CRozszerzanieOJedenPunkt} finitely many times we find a positive extension $\bar{F}$ of $F_m$ {\color{black}so that $\bar{F}^k(l_i,p_i)$ is well defined for all $i=1,...,m$. Since $\bar{F}$ is positive, we also have $\pi_2(\bar{F}^{k}(l_i,p_i))\geq \pi_2(\bar{F}^{k_i}(l_i,p_i))>M$} for every $i\leq m$.
\end{proof}

\begin{lemma}\label{UnionOfTwoCn}
Let $M\in{\color{black}\R}$. Let $h_0,h_1\in\mC_n$ be such that 
\[
\pi_2(\dom h_0\cup\dom h_1)<M<\pi_2(\img h_0\cup\img h_1).
\]
Then $h_0\cup h_1\in\mC_n$.
\end{lemma}

\begin{proof}
Since $\dom h_0\cap\dom h_1=\emptyset=\img h_0\cap\img h_1$, the function $h:=h_0\cup h_1$ is one-to-one. To prove that $h$ is order-preserving fix $(n,p),(k,q)\in\mC_n$. If $(n,p),(k,q)\in\dom h_i$, then $(n,p)\leq(k,q)\iff h_i(n,p)\leq h_i(k,q)\iff h(n,p)\leq h(k,q)$. If $(n,p)\in\dom h_0$ and $(k,q)\in\dom h_1$, then $p<M<q$ and therefore $(n,p)\leq(k,q)$. Since {\color{black}$\pi_2(h_0(n,p))<M<\pi_2(h_1(k,q))$}, then $h(n,p)\leq h(k,q)$. 
\end{proof}

\begin{lemma}(Key Lemma)\label{KeyLemmaCn}
Let ${\color{black}F\in\Ism_+(\mC_n)}${\color{black}, $G\in\Ism(\mC_n)$, $X$} be a finite subset of $\mC_n$ and {\color{black} $w(a,b)$ be an irreducible word}. Then there are a natural number $k\in{\color{black}\N}$, {\color{black} $\tilde{F}\in\Ism_+(\mC_n)$ and  $\tilde{G}\in\Ism(\mC_n)$ }such that
\begin{itemize}
\item[(i)] {\color{black}$F\subseteq\tilde{F}$, $G\subseteq\tilde{G}$};
\item[(ii)] for any ${\color{black}h_0,h_1}\in\Ism(\mC_n)$ such that $\dom h_0\cup\img h_0\subseteq X$ and $\dom h_1\cup\img h_1\subseteq\tilde{{\color{black}F}}^k(X)$, the function {\color{black}$h_0\cup h_1\in\Ism(\mC_n)$};
\item[(iii)] there is $y\in{\color{black}\mC_n}$ such that {\color{black}$w(\tilde{F},\tilde{G})(y)\neq y$}. 
\end{itemize}
\end{lemma}

\begin{proof}
Since ${\color{black}F}$ and ${\color{black}G}$ are partial isomorphisms, there is $M\in{\color{black}\R}$ such that 
\[
\pi_2(\dom {\color{black}F\cup\img F\cup \dom G\cup\img G}\cup X)<M. 
\]
By Lemma \ref{f5} there are {\color{black} $r,s\in\Ism_+(\Q)$} such that $\dom r\cup\dom s>M$ and $w(s,r)(p)\neq p$ for some rational $p>M$. {\color{black}Define $S(1,q):=(1,s(q))$, $q\in\on{dom}s$, and $R(1,q):=(1,r(q))$, $q\in\on{dom}r$, and let $F_1:=F\cup S$ and $G_1:=G\cup R$. Since $S,R\in\Ism_+(\mC_n)$,} by Lemma \ref{UnionOfTwoCn} we obtain that ${\color{black}F_1,G_1}\in\Ism(\B_\omega)$. Since ${\color{black}F,S}$ are positive, then so is ${\color{black}F_1}$. {\color{black}Also, $w(F_1,G_1)(1,p)=(1,w(s,r)(p))\neq (1,p)$.}

By Lemma \ref{CZbiorXWysokoDoGory} there are $k\in{\color{black}\N}$ and an extension {\color{black}$\tilde{F}$ of $F_1$} such that $\pi_2(\tilde{{\color{black}F}}^k(X))>M$.  By Lemma \ref{UnionOfTwoCn} we obtain {\color{black}(ii)}. Put ${\color{black}\tilde{G}=G}_1$ and observe that (i) and (iii) are fulfilled as well. 
\end{proof}
{\color{black}Now put 
$$
\Aut_+(\mC_n):=\{F\in\Aut(\mC_n):\forall_{(k,p)\in\mC_n}\;\pi_2(F(k,p))>p\}
$$
and observe that for $F\in\Aut(\mC_n)$, 
$$
F\in\Aut_+(\mC_n)\iff\forall{p\in\Q}\forall{k\leq n}\exists{p<q\in\Q}\exists{l\leq n}\;\;F(k,p)=(l,q).
$$
In particular, 
$$
\Aut_+(\mC_n)=\bigcap_{p\in\Q}\bigcap_{k\leq n}\bigcup_{p<q\in\Q}\bigcup_{l\leq n}\{F\in\Aut(\mC_n):F(k,p)=(l,q)\}.
$$
Since $\{F\in\Aut(\mC_n):F(k,p)=(l,q)\}$ is clopen, $\Aut_+(\mC_n)$ is $G_\delta$.

}

{\color{black}Observe that 
$$\Ism_+(\mC_n)=\Aut_+(\mC_n)^{<\omega}.$$}
The {\color{black}inclusion $\supset$} is clear. Assume that {\color{black}$F\in\Ism_+(\mC_n)$. Then} by Lemma \ref{CRozszerzanieOJedenPunkt} {\color{black}and} the back-and-forth argument we can inductively define $\tilde{F}\in\Aut_+(\mC_n)$ such that $F\subset\tilde{F}$. Thus $F\in\Aut_+(\mC_n)^{<\omega}$. 

Using Key Lemma and {\color{black} Theorem \ref{SecondGeneralTheorem}} we obtain the following.

\begin{corollary}
\begin{itemize}
\item[(i)] The set
\[
\{({\color{black}F,G})\in\Aut_+(\mC_n)\times\Aut(\mC_n):{\color{black}F}\text{ and }{\color{black}G}\text{ freely generate a dense subgroup of }\Aut(\B_\omega)\}
\]
is comeager in $\Aut_+(\mC_n)\times\Aut(\mC_n)$.
\item[(ii)] The set of cyclically dense elements $\bar{{\color{black}G}}\in\Aut(\mC_n)^m$ for the diagonal action in comeager in $\Aut(\mC_n)^m$. 
\end{itemize}
\end{corollary}

%%%%%%%%%%%%%%%%%%%%%%%%%%%%%%%%%%%%%%%%%%%%%%%%%%%%%%%%%%%%%%%%%%%%%%%%%%%%%%
%%
%% 				Generic poset D
%%
%%%%%%%%%%%%%%%%%%%%%%%%%%%%%%%%%%%%%%%%%%%%%%%%%%%%%%%%%%%%%%%%%%%%%%%%%%%%%%

\section{Generic poset $\D$}\label{SectionD}

Recall that {\color{black}the} generic poset $\D$, as a Fra\"{i}ss\'{e} limit of the family of all finite posets, is a unique countable \emph{existentially closed} (or simply \emph{e.c.}) poset in the sense that for any finite poset $X \subseteq \D$ and any its one-point extension $X\cup\{x\}$ (i.e. $(X\cup\{x\},\leq)$ is a poset such that {\color{black}$\leq_{|X}$} is the original order on $X$), there is $y\in\D$ and an isomorphism $f:X\cup\{x\}\to X\cup\{y\}$ which extends the identity on $X$. {\color{black}In other} words, if a finite poset $X\subseteq\D$ has one-point extension  $X\cup\{x\}$ where $x$ is some abstract element, then we may assume that $x$ belongs to $\D$. Moreover there are infinitely many $x$'s in $\D$ which can be used for this purpose. See eg.~\cite{Mekler} for more details. 

To make the construction we need also to introduce some technical properties. We say that $h \in \Ism(\D)$ is \emph{orbitally incomparable} if $h^k(x) \perp x$ for every $x \in \dom h$ and $k\geq 1$ with $h^k(x) \in \img h$. Note that although $\D$ is homogeneous, it may be not possible to extend some finite orbitally incomparable isomorphism to orbitally incomparable automorphism of $\D$. For an example consider set $A = \{a,b,c,d\} {\color{black}\subseteq} \D$, the relation $<_{|A} = \{(a,d), (b, c)\}$ and an isomorphism $h$ such that $h(a) = c$ and $h(b) = d$. Clearly, for any extension $h'$ of $h$ such that $c \in \dom h'$ we have $h'^2(a) = h'(c) > h'(b) = d > a$.

To make this property extendible let us define \emph{correctly} orbitally incomparable finite isomorphisms. Let $\sim$ be the equivalence relation on $\dom h \cup \img h$ given by condition $$x \sim y \quad \textrm{if and only if there is} \quad k \in \Z \quad \textrm{such that} \quad y = h^k(x).$$ The equivalence classes $[x]_h$ are precisely the orbits of $h.$ If it is clear which partial isomorphism is considered, we write simply $[x]$ for its equivalence class. 
\begin{definition} \label{def:coi}
Let $h \in \Ism(\D)$ be any finite isomorphism of $\D$. We say that $h$ is \emph{correctly orbitally incomparable} if and only if it is orbitally incomparable and the relation $\preceq$ on orbits of $h$ given by the condition
$$[x] \preceq [y] \quad \textrm{if and only if} \quad x' \le y', \textrm{~for some~} x' \in [x] \textrm{~and~} y' \in [y]$$ is a partial order.  
\end{definition}
If $h$ is a correctly orbitally incomparable isomorphism, then by $(\dom h\cup\img h)_\sim$ we denote the poset of all orbits $[x]_h$ with a partial order $\preceq$. For $x\in\D\setminus(\dom h\cup\img h)$ by $[x]$ we denote the singleton $\{x\}$.

{\bf A good $x$--extension.} Assume that $h\in\Ism(\D)$ is correctly orbitally incomparable and $x\in\D\setminus\dom h$. A \emph{good $x$--extension of $h$} is a function $h\cup\{(x,y)\}$ where $y\in\D\setminus(\dom h\cup\img h)$ is such that\\  
(G1) $y< z \iff x < h^{-1}(z)\text{ and }z < y \iff h^{-1}(z) < x$ for $z\in\img h$;\\
(G2) $z< y\iff (\exists v\in\img h\;\;\; z< v< y)$ for every $z\in\dom h\setminus\img h$;\\
(G3) $y< z\iff (\exists v\in\img h\;\;\; y< v< z)$ for every $z\in\dom h\setminus\img h$.

The following lemmas shows that it is possible to extend any correctly orbitally incomparable isomorphism in any possible way preserving the property. Lemmas~\ref{lem:orbitextension1}(a) and \ref{lem:neworbit} may be viewed as the equivalent for one step of back-and-forth method for extending finite isomorphisms. 
\begin{lemma} \label{lem:orbitextension1}
Let $h \in \Ism(\D)$ be any correctly orbitally incomparable isomorphism and {\color{black}$x\in\D$}. Assume one of the following
\begin{enumerate}
\item [(a)] $x \in \img h \setminus \dom h$;
\item [(b)] $x\notin\dom h\cup\img h$ and the poset $(\dom h\cup\img h)_\sim$ of all equivalence classes extended to $(\dom h\cup\img h)_\sim\cup\{[x]\}$ is a poset as well.
\end{enumerate}
 Then there exists $y\in\D\setminus(\dom h\cup\img h)$ such that $h'=h\cup\{(x,y)\}$ is a good $x$-extension, which in turn is
a correctly orbitally incomparable isomorphism. Moreover, the mapping $[v]_h\mapsto[v]_{h'}$ is the $\preceq$-isomorphism between $(\dom h\cup\img h)_\sim$ and $(\dom h'\cup\img h')_\sim$ in case (a), and between $(\dom h\cup\img h)_\sim\cup\{[x]\}$ and $(\dom h'\cup\img h')_\sim$ in case (b).
\end{lemma}

The relation $\preceq$ on the extended set of orbits $(\dom h\cup\img h)_\sim\cup\{[x]\}$ should be understood as follows. The relation $\preceq$ remains unchanged for $h$-orbits. If $[y]$ is an orbit of $h$, then $[y]\preceq[x]\iff(y'\leq x$ for some $y'\in[y])$ and similarly $[x]\preceq[y]\iff (x\leq y'$ for some $y'\in[y])$. 

\begin{proof}[Proof of Lemma~\ref{lem:orbitextension1}]
Let $h\in \Ism(\D)$ be any correctly orbitally incomparable isomorphism and $x\in\D\setminus\dom h$. Assume (a), that is  $x \in \img h \setminus \dom h$ (under the assumption (b) the proof is almost the same; below we describe slight differences). Let $X=\dom h\cup\img h$. Then $X$ is a finite subset of $\D$. Take some abstract element $y\notin \D$. We define a relation $\leq$ on $X\cup\{y\}$ which extends the order from $X$. Firstly we define this relation between $y$ and elements from $\img h$. For $z\in\img h$ put
\[
y < z \iff x < h^{-1}(z)\text{ and }z < y \iff h^{-1}(z)< x.
\]
Consequently $y\perp z\iff x\perp h^{-1}(z)$ for $z\in\img h$. Moreover, since $x\perp h^{-1}(x)$, then $y\perp x$ (under (b) $x$ is not in the $\img h$, and therefore the fact that $y\perp x$ need to be proved separately). Secondly we take a transitive closure of $\leq$, that is for $z\in\dom h\setminus\img h$ we put $z<y$ {\color{black}iff} $z < v$ and $v < y$ for some $v\in\img h$ and we put $y < z$ {\color{black}iff} $y < v$ and $v < z$ for some $v\in\img h$. Finally for those elements $z$ from $\dom h\setminus\img h$ for which we have put neither $z < y$ nor $y < z$, the relation $\leq$ remains unchanged, that is $y\perp z$. 

The extended relation is antisymmetric. Suppose not. There exists an element $z \in \dom h \cup \img h$ such that $z < y$ and $z > y$. If $z \in \img h$, we have $h^{-1}(z) < x$ and $h^{-1}(z) > x$, which yields a contradiction. If $z \in \dom h \setminus \img h$, there are two elements $v_1, v_2 \in \img h$ such that $z < v_1 < y$ and $z > v_2 > y.$ We have $v_2 < z < v_1$ and $h^{-1}(v_1) < x < h^{-1}(v_2)$. This contradicts the fact that $h$ is an order isomorphism. 

The extended relation is transitive. First, we prove that $z_1 < y$ and $y < z_2$ imply $z_1 < z_2$ for any $z_1, z_2 \in \dom h \cup \img h$. Consider four cases:
\begin{enumerate}
\item If $z_1,z_2\in\img h$, then $h^{-1}(z_1)< x < h^{-1}(z_2)$, which in turn implies that $z_1 < z_2$.
\item If $z_1\in\img h$ and $z_2\in\dom h\setminus\img h$, then there is $v\in\img h$ such that $y< v < z_2$. Thus $z_1<y<v$ and from case (i) we have $z_1<v$, and consequently $z_1<z_2$. %$h^{-1}(z_1) < x < h^{-1}(v)$, and therefore $z_1 < v$. Consequently $z_1 < z_2$.
\item If $z_1 \in \dom h \setminus \img h$ and $z_2 \in \img h$, then there exists $v \in \img h$ such that $z_1 < v < y$. Thus from case (i) we have $v < z_2$, and consequently $z_1 < z_2$.
\item If both $z_1, z_2 \in \dom h \setminus \img h$ then there are $v_1, v_2 \in \img h$ such that $z_1 < v_1 < y$ and $y < v_2 < z_2$. From the first case we have $v_1 < v_2$, and consequently $z_1 < z_2$.
\end{enumerate}

Now, we prove that $z_1 < z_2$ and $z_2 < y$ imply $z_1 < y$ for any $z_1, z_2 \in \dom h \cup \img h$ (the implication $(y < z_1$ and $z_1 < z_2) \implies y < z_2$ goes in the same way). Consider four cases:
\begin{enumerate}
\item If $z_1,z_2\in\img h$, then $h^{-1}(z_1) < h^{-1}(z_2) < x$. Thus $h^{-1}(z_1) < x$ and consequently $z_1 < y$.
\item If $z_1\in\img h$ and $z_2\in\dom h\setminus\img h$, then there is $v\in\img h$ such that $z_2 < v < y$. Then, as $z_1 < v$, using the previous case we obtain that $z_1 < y$.
\item If $z_1\in\dom h\setminus\img h$ and $z_2\in\img h$, then $z_2$ witnesses that $z_1 < y$.
\item If $z_1,z_2\in\dom h\setminus\img h$, then there is $v\in\img h$ with $z_2 < v < y$. Then $z_1 < v < y$, which implies that $z_1 < y$. 
\end{enumerate}

(It is the time to deal with the case (b). Here we assume that $x\perp y$ by definition. Clearly, the extended relation $\leq$ is antisymmetric. Suppose that there is $u\in\dom h\cup\img h$ which is between $x$ and $y$, say $x< u$ and $u < y$. If $u\in\img h$, then we obtain $h^{-1}(u) < x < u$ which contradicts the fact that $u\perp h^{-1}(u)$. If $u\in\dom h\setminus\img h$, then $u < y < h(u)$ which yields a contradiction as well. Similarly we deal with the case $y<u<x$. Hence, we can prove that the extended relation $\leq$ is transitive in the same way as before.)

As the extended relation $\le$ is a partial order, as $\D$ is existentially closed, we may assume that $y \in \D \setminus (\dom h \cup \img h)$. Thus $h' = h \cup \{(x, y\}$ is a good $x$--extension. The definition of extended relation $\le$ implies that $h'$ is order preserving. Now we prove that 
%\przemek{Przywracam częściowo usunięty fragment. Myślę, że trzeba wykazać, że dodany $y$ nie psuje relacji $\preceq$ i $h'$ będzie correctly oribitally incomparable.}
%$$y \perp h^{-k}(x) \ \textrm{for any} \ k \in \N \ \textrm{such that} \ h^{-k}(x) \in \dom h.$$

\begin{enumerate}
\item $y \perp h^{-k}(x)$ for any $k \in \N$ such that $h^{-k}(x) \in \dom h $;
\item $y > z$ implies $[x] \succ [z]$ for any $z \in \dom h \cup \img h$; 
\item $y < z$ implies $[x] \prec [z]$ for any $z \in \dom h \cup \img h$. 
\end{enumerate}

Suppose (i) this does not hold. Then there is $z \in [x]$ such that $\neg z \perp y$. Recall that $x\perp y$. If $z \in \img h\setminus\{x\}$ then $\neg h^{-1}(z) \perp x$ which contradicts the fact that $h$ is orbitally incomparable. Suppose that $z \in \dom h \setminus \img h$. Then $z< y$ or $y< z$. Suppose that $z < y$ (the opposite case is analogous). There is $v \in \img h$ with $z < v < y$. We have $[z] \prec [v]$ and, as $h^{-1}(v) < x$, $[v] \prec [x]$. But $[z] = [x]$, thus $\preceq$ is not antisymmetric, and consequently $h$ is not correctly orbitally incomparable, which yields a contradiction.  
(Here is the next difference in the proof if (b) is assumed. This paragraph is just not needed.)

Now we show (ii) and (iii), which means $h'$ is correctly orbitally incomparable. Assume that $y > z$ for some $z \in \dom h \cup \img h$. If $z \in \img h$, then $x > h^{-1}(z)$, and consequently $[x] \succ [z]$. If $z \in \dom h \setminus \img h$, there is $v \in \img h$ such that $y > v > z$. As $x > h^{-1}(v)$ we have $[x] \succ [v]$ and $[v] \succ [z]$ and hence, as $\preceq$ is a partial order $[x] \succ [z]$.
The symmetric case $y < z$ is analogous. 

To end the proof it is enough to show that $[u]_{h}\prec[v]_{h}$ iff $[u]_{h'}\prec[v]_{h'}$ (note that assuming $[x]_{h}:=[x]$ in the case (b), our proof works for this case as well). Since we have extended only the orbit $[x]_h$, the number of orbits have not changed and $[v]_h\prec[u]_h$ implies $[v]_{h'}\prec[u]_{h'}$ for every $u,v\in\dom h \cup \img h$. The opposite implication is clearly true for $u,v\notin[x]_{h'}$. Suppose that $[v]_{h'}\prec[x]_{h'}$ ($[x]_{h'}\prec[v]_{h'}$ goes similarly). There are $v'\in[v]_h$ and $x'\in[x]_h$ with $v' < x'$. If $x'\neq y$, then $[v]_h\prec[x]_h$. Otherwise $v'< y$. If $v'\in\img h$, then $h^{-1}(v')< x$ and $[v]_h\prec[x]_h$ as before. If $v'\in\dom h\setminus\img h$, then there is $u\in\img h$ with $v'< u< y$. By previous case $[u]_h\prec[x]_h$. Since $[v]_h\prec[u]_h$, then $[v]_h\prec[x]_h$.
\end{proof}

\begin{lemma}\label{lem:TransitivityFor2Points}
Let $h\in\Ism(\D)$ be correctly orbitally incomparable and let $x\in\D\setminus(\dom h\cup\img h)$. Let $[y]_h$ and $[z]_h$ be distinct orbits of $h$ such {\color{black}one of the following conditions holds}
\begin{enumerate}
\item $[y]_h\prec[z]_h$ and $[z]_h\prec[x]$; 
\item $[y]_h\succ[z]_h$ and $[z]_h\succ[x]$.
\end{enumerate}
 Then there is extension of $h$ to a correctly orbitally incomparable isomorphism $h'$ such that $\dom h'\cup\img h'=\dom h\cup\img h\cup[y]_{h'}$ and $[y]_{h'}\prec[x]$ provided (i) holds, and $[x]\prec[y]_{h'}$ provided (ii). Moreover the mapping $[v]_h\mapsto[v]_{h'}$ is the $\preceq$-isomorphism between $(\dom h\cup\img h)_\sim$ and $(\dom h'\cup\img h')_\sim$. 
\end{lemma}

The condition $\dom h'\cup\img h'=\dom h\cup\img h\cup[y]_{h'}$ means that to produce $h'$ from $h$ we extend only the orbit $[y]_h$; the other orbits remain unchanged. 

\begin{proof}
We prove the lemma under assumption \emph{(i)} only. The second one is symmetric. Note that $[y] \prec [z]$ and $[z] \prec [x]$ mean that there exist $i, j, k \in \Z$ such that $h^i(y) < h^j(z)$ and $h^k(z) < x$. The latter inequality implies that for any $m \in \Z$ such that $h^m(y) \in \dom h \cup \img h$ there is either $h^m(y) < x$ or $h^m(y) \perp x$. If $h^m(y) < x$ for some $m \in \Z$, then $[y] \prec [x]$. So suppose otherwise. Notice that if it is  $h^{i-j+k}(y) \in \dom h \cup \img h$ then $h^{i-j+k}(y) < h^k(z) < x$. Hence, it is possible to obtain $[y] \prec [x]$ extending the orbit of $y$ such that $h^{i-j+k}(y) \in \dom h \cup \img h$ holds. When $i-j+k \ge 0$ we use Lemma~\ref{lem:orbitextension1} directly (at most $i-j+k$ many times), otherwise we use it to $h^{-1}$ which is correctly orbitally incomparable isomorphism as well.

The moreover part of the assertion follows from Lemma~\ref{lem:orbitextension1}. 
\end{proof}

\begin{lemma} \label{lem:neworbit}
Let $h \in \Ism(\D)$ be any correctly orbitally incomparable isomorphism. Let $x \in \D \setminus (\img h \cup \dom h)$. There exists correctly orbitally incomparable isomorphism $h' \in \Ism(\D)$ such that $x \in \dom h'$ and $h\subset h'.$ 
\end{lemma}

\begin{proof}[Proof of Lemma~\ref{lem:neworbit}]
Let $h\in \Ism(\D)$ be any correctly orbitally incomparable isomorphism. Let $x \in \D \setminus (\img h \cup \dom h)$. The idea of the proof is to create temporary, single-element orbit $[x],$ extend some other orbits such way, to provide that the relation $\preceq$ is a partial order on $(\dom h\cup\img h)_\sim\cup\{[x]\}$ and, finally, extend orbit $[x]$ attaching $h'(x)$ by Lemma~\ref{lem:orbitextension1}. Note that $\dom h \cup \img h \cup \{x\}$ induce some finite poset.

Let us start with adding a singleton $[x] = \{x\}$ to the set of orbits of $h$ and extending the relation $\preceq$ on $[x]$. The extended relation is still antisymmetric. Suppose otherwise. There exists an element $z \in \dom h\cup \img h$ such that $[z] \prec [x]$ and $[z] \succ [x]$. It means that for some $i, j \in \Z$ there is $h^i(z) < x$ and $h^j(z) > x$. But it follows that $h^j(z) > h^i(z)$, a contrary.  

Analogously, as $[x]$ is a singleton, we get easily that for any $y, z \in \dom h \cup \img h$, if  $[y] \prec [x]$ and $[x] \prec [z]$, then $[y] \prec [z]$. However, $\preceq$ may be not transitive: when $[x] \prec [y]$ and $[y] \prec [z]$ it may be either $[x] \prec [z]$ (which is correct) or $[x] \perp [z]$ (which violates the condition). Similarly, when $[y] \prec [z]$ and $[z] \prec [x]$ it may be either $[y] \prec [x]$ or $[y] \perp [x]$. %Let us fix on the latter case (as both are analogous).

Let $\{([y_i]_h,[z_i]_h):i=1,2,\dots,l\}$ be an enumeration of all pairs of $h$-orbits such that either $[y]_h\prec[z]_h$, $[z]_h\prec[x]$ and $[y]_h\perp[x]$ or $[x]\prec[z]_h$, $[z]_h\prec[y]_h$ and $[x]\perp[y]_h$. Let $h_0:=h$. In the first step we extend $h_{0}$ to $h_1$ using Lemma \ref{lem:TransitivityFor2Points} for $h_0$, $[y_1]$ and $[z_1]$. By moreover part of Lemma \ref{lem:TransitivityFor2Points} the order $\preceq$ between $(\dom h_0\cup \img h_0)_\sim$ and $(\dom h_1\cup \img h_1)_\sim$ is isomorphic. This allows us to repeat the reasoning in the next $l-1$ steps. After $l$-th step we obtain the final extension $h_l$ of $h$. 

We show that $\preceq$ on $(\dom h_l\cup\img h_l)_\sim\cup\{[x]\}$ is transitive. Since $(\dom h\cup\img h)_\sim$ and $(\dom h_l\cup\img h_l)_\sim$ are order isomorphic, we need to check transitivity for orbits $[x],[y]_{h_l},[z]_{h_l}$ (for technical reasons, we assume $y,z\in \dom h\cup\img h$). The case when $[x]$ is $\preceq$-between $[y]_{h_l}$ and $[z]_{h_l}$ has been already checked. Suppose that $[y]_{h_l}\prec[z]_{h_l}$ and $[z]_{h_l}\prec[x]$ (the ``$\succ$'' case is analogous). Note that $[y]_h\prec[z]_h$. If $[z]_{h}\prec[x]$, then by the construction of $h_l$ we obtain $[y]_{h_l}\prec[x]$. If it is not true that $[z]_{h}\prec[x]$, then by the fact that $[z]_{h_l}\prec[x]$ there is $v$ such that $[z]_h\prec[v]_h$ and $[v]_h\prec[x]$. Thus $[y]_h\prec[v]_h$ and by construction of $h_l$, we obtain $[y]_{h_l}\prec[x]$ as in the previous case. 

Using Lemma~\ref{lem:orbitextension1}(b) we obtain correctly orbitally incomparable isomorphism $h'$ such that $h_l\subset h'$ and $x \in \dom(h')$. 
\end{proof}

\begin{lemma} \label{lem:incompextension}
Let $h \in \Ism(\D)$ be a correctly orbitally incomparable isomorphism. Let $A \subseteq \dom(h)$. There exists correctly orbitally incomparable isomorphism $g \in \Ism(\D)$ and $m \in \N$ such that $h\subseteq g$ and for any $x,y \in A$ {\color{black}it holds} $x \perp g^m(y)$. 
\end{lemma}
The proof is divided into few steps. Formally the above lemma is a corollary of Lemma~\ref{lem:incompextension2points}.
If $h$ is a correctly orbitally incomparable isomorphism and $[u]$ is its orbit, then by the \emph{length of the orbit $[u]$} we mean the number $k\in\N$ so that $h^k(u)\in\img h\setminus \dom h$, provided we chosen the representative $u\in\dom h\setminus \rng h$. Let the length of the orbit be denoted by $|[u]|$.
\begin{lemma}\label{lem:firststep1}
Let $h \in \Ism(\D)$ be a correctly orbitally incomparable isomorphism. Let $x,y\in\dom h\setminus\img h$ be such that $\lnot([x]\perp[y])$ and there is no intermediate $u \in \dom h$ between $[x]$ and $[y]$. Let $k\geq 1$ be the length of the orbit $[x]$. Then $h_0^{k+1}(x)\perp y$ for each good $h^k(x)$--extension $h_0$ of $h$. 
\end{lemma}
\begin{proof}
Assume that $[x]_h\preceq [y]_h$. We deal with the opposite cases similarly.
Let $h_0$ be a good $h^k(x)$--extension of $h$. Suppose that $h_0^{k+1}(x)$ is comparable with $y$. By Lemma \ref{lem:orbitextension1} $[x]_{h_0}\prec[y]_{h_{0}}$, and therefore $h_0^{k+1}(x) < y$. By the definition of good extension there is $v\in\img h$ such that $h_0^{k+1}(x) < v < y$. Thus $[x]_{h_0}\prec[v]_{h_0}\prec[y]_{h_0}$. By Lemma \ref{lem:orbitextension1} we have $[x]_{h}\prec[v]_{h}\prec[y]_{h}$. A contradiction.
\end{proof}

We say that $h'$ is a \emph{good orbit extension} of $h$ if there are $x_1,\dots,x_l$ and $h_1,\dots,h_l$ such that
\begin{enumerate}
\item $x_i\in\img h_{i-1}\setminus\dom h_{i-1}$, where $h_0=h$;
\item $h_i$ is a good $x_i$-extension of $h_{i-1}$ and $h'=h_l$. 
\end{enumerate}
Note that a good orbit extension of $h$ extends only existing orbits and does not add new ones. Also it is worth to observe that if $y\in\dom h\setminus\img h$, then $y\in\dom h_i\setminus \img h_i$ at each step. By the definition of good orbit extension we immediately obtain the following strengthening of Lemma \ref{lem:firststep1}.
\begin{lemma}\label{lem:firststep2}
Let $h \in \Ism(\D)$ be a correctly orbitally incomparable isomorphism. Let $x,y\in\dom h\setminus\img h$ be such that $\lnot([x]_h\perp[y]_h)$ and there is no intermediate $[u]_h$ between $[x]_h$ and $[y]_h$. Let $k\geq 1$ be the length of the orbit $[x]_h$ and let $g$ be a good orbit extension of $h$. Assume that for some $m\geq 1$, the length of the orbit $[x]_{g}$ equals $k+m$. Then $g^{k+i}(x)\perp y$ for $1\leq i\leq m$. 
Consequently, if $l$ is the length of the orbit $[y]_g$, then $g^{k+j+i}(x)\perp g^j(y)$ for all $0\leq j\leq \min\{l,m-1\}$ and $1\leq i\leq m-j$.
\end{lemma}

For a correctly orbitally incomparable $h\in\Ism(\D)$ and $x,y\in\dom h\cup\img h$ with $\lnot([x]\perp[y])$ by \emph{$[x][y]$-chain} we understand a sequence $[z_0],[z_1],\dots,[z_j]$ such that
\begin{enumerate}
\item $[z_0]=[x]$ and $[z_j]=[y]$;
\item $[z_0]\prec[z_1]\prec\dots\prec[z_j]$ or $[z_0]\succ[z_1]\succ\dots\succ[z_j]$;
\item for each $i=1,2,\dots,j$ there is no intermediate element $[u]$ between $[z_{i-1}]$ and $[z_{i}]$.
\end{enumerate}
For technical reasons, we assume that representative $z$ of $[z]$ belongs to $\dom h\setminus \img h$.
By the length of $[x][y]$-chain $[z_0],[z_1],\dots,[z_j]$ we mean the number $j$ (if $[x] = [y]$ the length is $0$). By $\rho([x],[y])$ denote the maximum of the lengths of all $[x][y]$-chains (for completeness, we can put $\rho([x],[y])=\infty$ if $[x]\perp[y]$). Note that if $g$ is a good orbit extension of $h$, then $[x]_h\mapsto[x]_g$ is an isomorphism of posets $(\dom h\cup\img h)_\sim$ and $(\dom g\cup\img g)_\sim$. Thus good orbit extensions do not affect chains, their lengths and the function $\rho$.   

\begin{lemma}
Let $h \in \Ism(\D)$ be a correctly orbitally incomparable isomorphism. Let $x,y\in\dom h\setminus\img h$ be such that $\lnot([x]_h\perp[y]_h)$. There is a natural number $m_{xy}$ such that for every good orbit extension $g$ of $h$,
\[
g^k(x)\perp y\text{ for }k > m_{xy}\text{ provided }g^k(x)\in\img g.
\]
\end{lemma}

\begin{proof}
Let $h \in \Ism(\D)$ be a correctly orbitally incomparable isomorphism. %let $g \in \Ism(\D)$ be any good orbit extension of $h$ with sufficiently long orbits. Let $x,y\in\dom h\setminus\img h$ be such that $[x]\preceq[y]$. 
Define $\{m_{u,v}:u, v \in \dom h \setminus \img h,\;\lnot(u\perp u)\}$ according to the following formula:% let $n_{uv} \in \N$ be defined as follows
\[
m_{uv} := \left\{ 
\begin{array}{ll}
0, & \text{for~} u = v; \\
|[u]_h|, & \text{for~} \rho([u],[v]) = 1; \\
\max \left\{ \sum_{i = 0}^{j-1} m_{z_{i}z_{i+1}} \colon [z_0], [z_1], \dots, [z_j] \textrm{~is a~} [u][v]\textrm{-chain}\right\}, & \text{otherwise.}
\end{array}
\right.
\]
We show that $m_{u,v}$ satisfies our needs. The proof is inductive with respect to the length $\rho([x],[y])$. When $\rho([x], [y]) = 0$, in other words $[x] = [y]$, the result follows from orbital incomparability of $h$. If $\rho([x],[y])=1$, that is if there is no intermediate element $[u]$ between $[x]$ and $[y]$, then we are done by Lemma \ref{lem:firststep2}. Let $i\geq 1$, $\rho([x], [y])=i+1$ and suppose that for every $u,v \in \dom h \setminus \img h$ with $\rho([u],[v])\leq i$ there is $g^k(u) \perp v$ for $k > m_{uv}$. % where $m_{uv} \le n_{uv}$.
Now suppose that $g^{k}(x)$ is comparable with $y$ for some $k > m_{xy}$. % {\color{black}that is, $g^k(x)<y$.} 
 By the definition of a good extension, there is some $w \in \img g$ which lies between $g^{k}(x)$ and $y$ (indeed, if $h_{i-1}$ is such that $g^{k-1}(x)\in \img h_{i-1}\setminus \dom h_{i-1}$, then $y\in \dom h_{i-1}\setminus \img h_{i-1}$). Let $w = g^j(z)$ for some $z \in \dom h \setminus \img h$ and some $j \in \N$. Clearly, the orbit $[z]$ belongs to some $[x][y]$-chain, while $\rho([x][z]) \le i$ and $\rho([z][y]) \le i$.
Assume first $k\geq j$. Because $\lnot(g^{k}(x) \perp g^j(z))$ one has $\lnot(g^{k-j}(x) \perp z)$, hence $k - j \le m_{xz}$. If $k<j$, we also have $k-j\le m_{xz}$. Moreover, $\lnot(g^j(z) \perp y)$ implies $j \le m_{zy}$. Therefore, one has $k \le m_{xz} + m_{zy} \le m_{xy}$. This is a contradiction. 
\end{proof}

\begin{corollary} \label{lem:incompextension2points}
Let $h \in \Ism(\D)$ be a correctly orbitally incomparable isomorphism. Let $x,y\in\dom h$. There is $m_{xy}\in\N$ such that for any good orbit extension $g$ of $h$ and any $i\geq 1$, 
\[
g^{m_{xy}+i}(x)\perp y \text{ provided }g^{m_{xy}+i}(x)\in\img g. 
\] 
\end{corollary}
\begin{proof}
If $x,y\in\dom h\setminus\img h$ and $\lnot([x]\perp[y])$, then we choose $m_{xy}$ as in previous lemma. In general, we take
$$
m_{xy}:=\left\{\begin{array}{ccc}m_{x'y'}+|[y]|&\mbox{if}&\lnot([x]\perp[y])\;\mbox{and}\;x'\in[x],y'\in[y],\;x',y'\in\dom h\setminus\rng h;\\
0&\mbox{if}&[x]\perp[y].\end{array}\right.
$$
\end{proof}

\begin{proof}[Proof of Lemma~\ref{lem:incompextension}]
Using Corollary~\ref{lem:incompextension2points} clearly, it is enough to take
\[
m := \max \{m_{xy} \colon x,y \in A\}
\]
and to take any good orbit extension $g$ producing long enough orbits for elements of the set $A$ (the existence of $g$ is guaranteed by Lemma \ref{lem:orbitextension1}). 
\end{proof}

For $A,B\subseteq\D$ we write $A\perp B$ if $a\perp b$ for every $a\in A$ and $b\in B$. 
\begin{lemma}\label{lem:GluingTwoIsomorphisms}
Let $h_0,h_1\in\Ism(\D)$ be such that 
\[
\dom h_0\cup\img h_0\perp \dom h_1\cup\img h_1.
\]
Then $h_0\cup h_1\in\Ism(\D)$. Moreover, if $h_0$ and $h_1$ are correctly orbitally incomparable, then so is $h_0\cup h_1$. 
\end{lemma}

\begin{proof}
Since $\dom h_0\cap\dom h_1=\emptyset=\img h_0\cap\img h_1$, the function $h:=h_0\cup h_1$ is one-to-one. We need only to prove that $h$ is order-preserving. If $a,b\in\dom h_i$, then
\[
a\leq b\iff h_{\color{black}i}(a)\leq h_i(b)\iff h(a)\leq h(b).
\]
If $a\in\dom h_0$ and $b\in\dom h_1$, then $a\perp b$ and $h(a)\perp h(b)$.\\
Now if $[a]$ is an orbit of $h_0$ and $[b]$ is an orbit of $h_1$, then $[a]\perp[b]$. Hence if $(\dom h_0\cup\rng h_0)_\sim$ and $(\dom h_1\cup\rng h_1)_\sim$ are posets, then so is $(\dom (h_0\cup h_1)\cup\img(h_0\cup h_1))_\sim$.
\end{proof}

The correct orbital incomparability can be also defined for automorphisms of $\D$. Automorphism $f\in\Aut(\D)$ is called \emph{orbitally incomparable} if $f^k(x)\perp x$ for every $x\in\D$ and every $k\in\Z\setminus\{0\}$. By $[x]$ we denote the $x$-orbit $\{f^k(x):k\in\Z\}$ of $f$. A function $f\in\Aut(\D)$ is called \emph{correctly orbitally incomparable} if $f$ is orbitally incomparable and the relation $\preceq$ on orbits of $f$ given by $[x]\preceq[y]$ if and only if $x'\leq y'$ for some $x'\in[x]$ and $y'\in[y]$ is a partial order. By $\mathcal{X}$ denote the set of all correctly orbitally incomparable automorphisms of $\D$.

We are very close to formulate the "Key Lemma". It can be easily seen that $\mX^{<}$ contains isomorphisms which may not be correctly orbitally incomparable. However, the following holds:
\begin{lemma}\label{aabb}
(1) If $f\in\Ism(\D)$ is correctly orbitally incomparable, then $f\in\mX^{<}$.\\
(2) For any $f\in\mX^{<}$ there exists correctly orbitally incomparable $f'\in\Ism(\D)$ such that $f\subseteq f'$.
\end{lemma}
\begin{proof}
The point (1) follows from Lemma \ref{lem:neworbit} and the back-and-forth argument (once again we remark that if $h$ is correctly orbitally incomparable, then so is $h^{-1}$).\\
We prove (2).\\
Clearly, $f$ is orbitally incomparable, hence there are distinct $x_1,...,x_n\in\dom f\setminus \rng f$ such that $[x_1]_f,...,[x_n]_f$ are all orbits of $f$.
 Let $\tilde{f}\in\mX$ be such that $f\subset \tilde{f}$. For each $i,j$ with $[x_i]_{\tilde{f}}<[x_j]_{\tilde{f}}$, choose $x_{ij}\in[x_i]_{\tilde{f}},\;x_{ji}\in[x_j]_{\tilde{f}}$ so that $x_{ij}<x_{ji}$. For $i=1,...,n$, let $C_i=\{j=1,...,n:\lnot([x_i]_{\tilde{f}}\perp[x_j]_{\tilde{f}})\}$.\\
Finally, for every $i=1,...,n$, let $x_i'$ be such that for every $y\in[x_i]_f\cup\{x_{ij}:j\in C_i\}$, there is $k>0$ so that $y=\tilde{f}^k(x'_i)$, and set
%\begin{itemize}
%\item[(a)] for every $y\in[x_i]_f$ there is $k>0$ such that $y=\tilde{f}^k(x_i')$;
%\item[(b)] for $j\in C_i$, there $k>0$, $x_{ij}=f^k(x_i')$.
%\end{itemize}
$$k_i:=\max\{k:\tilde{f}^k(x_i')\in [x_i]_{f}\cup\{x_{ij}:j\in C_i\}\},$$
$$
A_i:=\{\tilde{f}^k(x_i'):k=0,...,k_i\}
$$
and let $f':=\tilde{f}_{\vert A_1\cup...\cup A_n}$. Then $f\subset f'$, $f'\in \Ism(\D)$. Directly from the construction we see that for all $i,j=1,...,n$, $[x'_i]_{f'}<[x'_j]_{f'}$ iff $[x'_i]_{\tilde{f}}<[x'_j]_{\tilde{f}}$. Since $[x'_i]_{f'}$ are all orbits of $f'$ and $\tilde{f}$ is correctly orbitally incomparable, so is $f'$.% and $f'$ is correctly orbitally incomparable. The latter follows from the fact that for all $i,j=1,...,n$, , which can be easily seen by the construction of sets $A_i$.
\end{proof}
Now we can formulate the "Key Lemma":

\begin{lemma}[Key Lemma]
Let $f \in \mX^{<}$, $g\in\Ism(\D)$, let $X$ be a finite subset of $\D$ and let $w(a,b)$ be a word of two letters. Then there are $k\in\N$, correctly orbitally incomparable $\tilde{f} \in \Ism(\D)$, and $\tilde{g} \in \Ism(\D)$ such that
\begin{enumerate}
\item $f\subseteq\tilde{f}$, $g\subseteq\tilde{g}$ and $X\subset\dom \tilde{f}^k$;
\item for any two $h_0,h_1\in\Ism(\D)$ such that $\dom h_0\cup\img h_0\subseteq X$ and $\dom h_1\cup\img h_1\subseteq \tilde{f}^k(X)$ the union $h_0\cup h_1$ belongs to $\Ism(\D)$;
\item there is $y\in\D$ such that $w(\tilde{f},\tilde{g})(y)\neq y$.
\end{enumerate}
\end{lemma}

\begin{proof}
By Lemmas \ref{lem:neworbit} and \ref{aabb} we may assume that $X\subset\dom f$ and $f$ is correctly orbitally incomparable. Using Lemma \ref{lem:incompextension} we find large enough $k$ and a correctly orbitally incomparable $f_0$ such that $f\subseteq f_0$, $X\subset \dom f_0^k$ and $X\perp f_0^k(X)$. By Lemma \ref{lem:GluingTwoIsomorphisms} we obtain (ii). Assume that $w(a,b)$ is a word of the form $a^{m_k}b^{n_k}\dots a^{m_0}b^{n_0}$ where $n_0,m_k\in\Z$ and $n_k,\dots, m_0\in\Z\setminus\{0\}$. Let $Y=\{1,2\dots,m\}$ where $m=\sum_{i=0}^k(|n_i|+|m_i|)+1$. By Lemma \ref{WordLemma} there are $A,B\subset Y$ and functions $u:A\to Y$ and $v:B\to Y$ such that $u,v$ have no cycles and $w(u,v)(1)=m$. Using e.c. property for $\D$ we find $x_1,\dots,x_m$ such that $x_i\perp x_j$ for $i\neq j$ and $\{x_1,\dots,x_m\}\perp {\color{black}\dom f_0\cup\img f_0}\cup\dom g\cup\img g$. Let $f_1:\{x_i:i\in A\}\to\{x_1,\dots,x_m\}$ and $g_1:\{x_i:i\in B\}\to\{x_1,\dots,x_m\}$ be given by the formulas $f_1(x_i)=x_{u(i)}$ and $g_1(x_i)=x_{v(i)}$. Then $w(f_1,g_1)(x_1)=x_m$. Note that $f_1$ is correctly orbitally incomparable. By Lemma \ref{lem:GluingTwoIsomorphisms} $f_0\cup f_1$ is correctly orbitally incomparable and $g\cup g_1\in\Ism(\D)$. Finally, (i)-(iii) are satisfied by $\tilde{f}:=f_0\cup f_1$ and $\tilde{g}:=g\cup g_1$.% Therefore we have (iii). {\color{black}Finally, (i) also holds by the construction}.
\end{proof}

We also need the following:
\begin{lemma}\label{DGDelta}
$\mathcal{X}$ is a $G_\delta$ subset of $\Aut(\D)$. 
\end{lemma}

\begin{proof}
Note that {\color{black}$f\in\mX$ iff the following holds:
\begin{itemize}
\item[(a)] $\forall_{x\in\D}\;\forall_{k>0}\;(f^k(x)\perp x)$;
\item[(b)] $\forall_{x,y,z\in\D}\;([x]\preceq[y]\;\mbox{and}\;[y]\preceq[z])\Rightarrow([x]\preceq[z])$;
\item[(c)] $\forall_{x,y\in\D}\;([x]\preceq[y]\;\mbox{and}\;[y]\preceq[x])\Rightarrow([x]=[y])$.
\end{itemize}
Hence it is enough to show that sets
$$
A:=\{f\in\Aut(\D):\forall_{x\in\D}\;\forall_{k>0}\;(f^k(x)\perp x)\}
$$
$$
B:=\{f\in\Aut(\D):\forall_{x,y,z\in\D}\;([x]\preceq[y]\;\mbox{and}\;[y]\preceq[z])\Rightarrow([x]\preceq[z])\}
$$
$$
C:=\{f\in\Aut(\D):\forall_{x,y\in\D}\;([x]\preceq[y]\;\mbox{and}\;[y]\preceq[x])\Rightarrow([x]=[y])\}
$$
are $G_\delta$ (then its intersection $A\cap B\cap C=\mX$ is also $G_\delta$).\\
We first deal with $A$.}
Note that 
\[
f^k(x)\perp x\iff\exists a_1,a_2,\dots,a_k\in\D\;(f(x)=a_1,f(a_1)=a_2,\dots,f(a_{k-1})=a_k\text{ and }x\perp a_k).
\]
Thus
\[
\{f\in\Aut(\D):f^k(x)\perp x\}=\bigcup_{a_1,\dots,a_k\in\D,\;x\perp a_k}\{f\in\Aut(\D):f(a_1)=a_2,\dots,f(a_{k-1})=a_k\}
\]
is {\color{black}open, as a union of open sets}. {\color{black} Also
$$
A=\bigcap_{x\in\D}\bigcap_{k>0}\{f\in\Aut(\D):f^k(x)\perp x\}
$$
so it is} $G_\delta$ in $\Aut(\D)$. {\color{black}Before we deal with $B$ and $C$, observe that for any $x,y\in\D$, the sets
$$
D_{x\leq y}:=\{f\in\Aut(\D):[x]\preceq[y]\}
$$
$$
D_{x=y}:=\{f\in\Aut(\D):[x]=[y]\}
$$
are open. We just shows it for the first set. The case of the second goes in the same manner. If $x\leq y$, then $D_{x\leq y}=\Aut(\D)$. In the converse case, we have ($[x]\preceq[y]$ iff $\exists_{k>0}\;(f^k(x)\leq y)\;\vee\;(x\leq f^k(y))$), and 
\[
f^k(x)\leq y\iff\exists a_1,a_2,\dots,a_k\in\D\;(f(x)=a_1,f(a_1)=a_2,\dots,f(a_{k-1})=a_k\text{ and }a_k\leq y),
\]
\[
x\leq f^k(y)\iff\exists a_1,a_2,\dots,a_k\in\D\;(f(y)=a_1,f(a_1)=a_2,\dots,f(a_{k-1})=a_k\text{ and }x\leq a_k),
\]
hence
$$
D_{x\leq y}=\bigcup_{a_1,...,a_k\in\D\;a_k\leq y}\;\{f\in\Aut(\D):f(x)=a_1,...,f(a_{k-1})=a_k\}\cup$$ $$\cup \bigcup_{a_1,...,a_k\in\D\;x\leq a_k}\;\{f\in\Aut(\D):f(y)=a_1,...,f(a_{k-1})=a_k\}.
$$
Thus $D_{x\leq y}$ is open in this case also.\\
Now observe that for all $x,y,z\in\D$ and $f\in\Aut(\D)$, 
$$
([x]\preceq[y]\;\mbox{and}\;[y]\preceq[z])\;\Rightarrow\;[x]\preceq[z])\;\iff
\lnot([x]\preceq[y])\;\mbox{or}\;\lnot([y]\preceq[z])\;\mbox{or}\;[x]\leq [z].
$$
Hence 
$$
B=\bigcap_{x,y,z\in\D}\left((\Aut(\D)\setminus D_{x\leq y})\cup(\Aut(\D)\setminus D_{y\leq z})\cup D_{x\leq z}\right)
$$
is $G_\delta$ as countable intersection of $G_\delta$-sets (we remark that closed sets are $G_\delta$ as we work in metric spaces).\\
In the same way we can show that
$$
C=\bigcap_{x,y\in\D}\left((\Aut(\D)\setminus D_{x\leq y})\cup(\Aut(\D)\setminus D_{y\leq x})\cup D_{x=y}\right)
$$
so $C$ is $G_\delta$ also.
}
%\[
%\Big(([x]\preceq[y]\text{ and }[y]\preceq[x])\implies[x]=[y]\Big)\iff
%\]
%\[
%\Big((\exists n\in\Z(f^n(x)\leq y)\text{ and }\exists k\in\Z(f^k(y)\leq x))\implies \exists m\in\Z(f^m(x)=y)\Big)\iff
%\]
%\[
%\Big((\forall n\in\Z(f^n(x)\nleq y)\text{ or }\forall k\in\Z(f^k(y)\nleq x))\text{ and }\exists m\in\Z(f^m(x)=y)\Big).
%\]
%Using the same method as before one can show that the sets
%\[
%\{f\in\Aut(\D):f^k(x)\leq y\}\text{, }\{f\in\Aut(\D):f^k(x)\nleq y\}\text{ and }\{f\in\Aut(\D):f^k(x)=y\}
%\]
%are clopen. Therefore the set 
%\[
%\{f\in\Aut(\D):([x]_f\preceq[y]_f\text{ and }[y]_f\preceq[x]_f)\implies[x]_f=[y]_f\}
%\]
%is a $G_\delta$ subset of $\Aut(\D)$. Thus
%\[
%\{f\in\Aut(\D):\preceq\text{ is antisymmetric on obrits of }f\}=
%\]
%\[
%\bigcap_{x,y,z\in\D}\{f\in\Aut(\D):([x]_f\preceq[y]_f\text{ and }[y]_f\preceq[x]_f)\implies[x]_f=[y]_f\}
%\]
%is a $G_\delta$ subset of $\Aut(\D)$ as well. Using the very similar argument we obtain that 
%\[
%\{f\in\Aut(\D):\preceq\text{ is transitive on obrits of }f\}
%\]
%is a $G_\delta$ subset of $\Aut(\D)$. Finally we obtain that $\mathcal{X}$ is a $G_\delta$ subset of $\Aut(\D)$. 
\end{proof}

Using Key Lemma, Lemma \ref{DGDelta} and {\color{black}Theorem \ref{SecondGeneralTheorem}} we obtain the following.

\begin{corollary}
\begin{itemize}
\item[(i)] The set
\[
\{(f,g)\in\mathcal{X}\times\Aut(\D):f\text{ and }g\text{ freely generate a dense subgroup of }\Aut(\D)\}
\]
is comeager in $\mathcal{X}\times\Aut(\D)$.
\item[(ii)] {\color{black}For any $m\in\N$, }the set of cyclically dense elements $\bar{g}\in\Aut(\D)^m$ for the diagonal action is comeager in $\Aut(\D)^m$. 
\end{itemize}
\end{corollary}

Glass, McCleary and Rubin proved that $\Aut(\D)$ is freely topologically two generated, see \cite[Proposition 4.1]{GMR}. Our proof is direct and we use quite different methods than that in \cite{GMR}, and moreover, we obtain a stronger assertion. By \cite{KT}, $\Aut(\D)$ has strong Rokhlin property. It is still unknown if $\Aut(\D)$ has ample generics -- Truss conjectured in \cite{T2007} that, as in $\Aut(\Q)$, it is not the case.

\section{Remarks}

\begin{proposition}\label{PropNowhereDense}
Let $A$ be a countable structure. Assume that for any pair $f_0,g_0\in\Ism(A)$ there is a finite set $X\subset A$ and $f_1,g_1{\color{black},h}\in\Ism(A)$ such that\\
(1) $f_0\subset f_1$, $g_0\subset g_1$;\\
(2) $X\subset\dom f_1\cap\dom g_1$;\\
(3) $f_1(X)=X=g_1(X)$;\\
(4) {\color{black}$X\subset\dom h$} and $h(X)\nsubseteq X$. \\
Then the set
\[
\{(f,g)\in\Aut(A)\times\Aut(A):\langle f,g\rangle\text{ is dense in }\Aut(A)\}
\]
is {\color{black} nowhere dense} in $\Aut(A)\times\Aut(A)$.
\end{proposition}

\begin{proof}
Let $f_0,g_0\in\Ism(A)$. Find $X\subset A$ and $f_1,g_1,h\in\Ism(A)$ fulfilling (1)--(4). Note that for any extensions {\color{black}$f$} and  {\color{black}$g$} of $f_1$ and $g_1$, respectively, and any word $w(a,b)$ we have {\color{black}$w(f,g)(X)=X$}. Therefore {\color{black}$\tilde{h}\nsubseteq w(f,g)$ for any extension $\tilde{h}$ of $h$, so $\langle f,g\rangle$ is not dense in $\Aut(A)$}. {\color{black}Hence
$$
\{(f,g)\in\Aut(A)\times\Aut(A):f_1\subset f,\;g_1\subset g\}\subseteq$$ $$\subseteq \{(f,g)\in\Aut(A)\times\Aut(A):f_0\subset f,\;g_0\subset g\}\cap\{(f,g)\in\Aut(A)\times\Aut(A):\langle f,g\rangle\text{  {\bf is not} dense in }\Aut(A)\}
$$
Hence the set 
\[
\{(f,g)\in\Aut(A)\times\Aut(A):\langle f,g\rangle\text{  {\bf is not} dense in }\Aut(A)\}
\]
contains open dense set} in $\Aut(A)\times\Aut(A)$.
\end{proof}
\begin{remark}\emph{
In the previous sections {\color{black}we proved} that sets of the form
\[
\{(f,g)\in{\color{black}\mathcal{X}}:\langle f,g\rangle\text{ is dense in }\Aut(A)\}
\]
are comeager in ${\color{black}\mathcal{X}}$ where $A$ is a countable ultrahomogeneous poset and $\mathcal{X}$ is {\color{black}certain} $G_\delta$-subset of ${\color{black}\Aut(A)\times\Aut(A)}$. A set $\mathcal{X}$ need to be meager in ${\color{black}\Aut(A)\times\Aut(A)}$. To prove it we can use Proposition \ref{PropNowhereDense}. \\
If $f,g\in\Ism(\B_n)$ {\color{black}(where $n<\omega$)}, then there is $p\in\Q$ such that $\dom f\cup \img f\cup\dom g\cup\img g<p$. Let $X:=\{p\}\times\{1,2,\dots,n\}$. Define $u(p,k)=(p,\tau_f(k))$ and $v(p,k)=(p,\tau_g(k))$ for $k\leq n$. Then $f\cup u,g\cup v\in\Ism(A)$. Put $h(p,1)=(p+1,1)$. Clearly the assumptions {\color{black}from} Proposition \ref{PropNowhereDense} are fulfilled. The similar argument works for $\B_\omega$, $\mC_n$ and $\mC_\omega$. \\
If $f,g\in\Ism(\D)$, then there are $x,y$ such that $x\perp y$ and $x,y$ are incomparable to any element from $\dom f\cup \img f\cup\dom g\cup\img g$. Define $u(x)=x=v(x)$ and $u(y)=y=v(y)$. Then $f\cup u,g\cup v\in\Ism(A)$. Put $h(x)=y$. As before the assumptions {\color{black}from} Proposition \ref{PropNowhereDense} are fulfilled. 
}\end{remark}
\begin{remark}\label{BnHasNoCyclicallyDenseElements}\emph{
There is an asymmetry between Section \ref{SectionBn} and Sections \ref{SectionBOmega}--\ref{SectionC}. We were able to prove stronger results for $\B_\omega$, $\mC_n$ and $\mC_\omega$ using simpler arguments than that for $\B_n$. There is a structural reason for that. Firstly let us note that there is no cyclically dense element $g\in\B_n$  {\color{black}for $n\geq 2$, for the diagonal action $\B_n$ on $\B_n$ (and, consequently, no cyclically dense $\bar{g}\in\B_n^m$)}. Suppose that $g\in\B_n$ is cyclically dense. Then there is $f\in\B_n$ such that $\{f^kgf^{-k}:k\in\Z\}$ is dense in $\B_n$. Thus $\{\tau^k_f\tau_h\tau^{-k}_f:k\in\Z\}$ equals $S_n$. In particular there is $k\in\Z$ such that $\tau^k_f\tau_h\tau^{-k}_f=\on{id}$. Thus $\tau_{\color{black}h}=\on{id}$ and consequently $S_n=\{\on{id}\}$ which means that $n=1$. This is a contradiction. 
}\end{remark}
\begin{remark}\emph{
The complexity of proofs in Section \ref{SectionBn} have also the other cause. Let us introduce the following notion. Let $G$ be a finite group which is generated by two elements. {\color{black}\emph{The complexity}} $c_G$ of $G$ is the smallest number $k$ such that there are $a,b\in G$ which generates $G$ and 
\[
G=\{a^{n_1}b^{m_1}a^{n_2}b^{m_2}\dots a^{n_k}b^{m_k}:n_i,m_i\in\Z\}. 
\]
By $c_n$ we denote the complexity of $S_n$. The Landau's number $L_n$ is the maximum order of permutation from $S_n$. It turns out that $\log L_n\sim \sqrt{n\log n}$, see \cite{Mil}. On the other hand Stirling's formula says that $\vert S_n\vert=n!\sim \sqrt{2\pi n}(\frac{n}{e})^n$.  Using these facts we prove that $c_n\to\infty$. \\
Suppose that $c_n$ does not tend to infinity. That means that $c_n\leq k$ for {\color{black}some $k$ and} every $n\in\N$. Let $a,b\in S_n$ be such that
\[
S_n=\{a^{n_1}b^{m_1}a^{n_2}b^{m_2}\dots a^{n_k}b^{m_k}:n_i,m_i\in\Z\}. 
\]
Then $n!\leq L_n^{2k}$. Using approximations for Landau's number $L_n$ and for $n!$ we obtain from the latter inequality that $n$ is bounded from above. That yields a contradiction.\\
This shows that the word which we have constructed in the proof of Lemma \ref{KillingIsom} cannot be short and simple.
}\end{remark}

%%%%%%%%%%%%%%%%%%%%%%%%%%%%%%%%%%%%%%%%%%%%%%%%%%%%%%%%%%%%%%%%%%%%%%%%%%%%%%
%%
%% Bibliografia
%%
%%%%%%%%%%%%%%%%%%%%%%%%%%%%%%%%%%%%%%%%%%%%%%%%%%%%%%%%%%%%%%%%%%%%%%%%%%%%%%

\end{document}